\documentclass[a4paper,12pt]{article}

\usepackage{epsfig}
\usepackage{amsfonts,eucal}
\usepackage{amssymb,amsmath,amsthm}
\usepackage[dvipsnames]{xcolor}
\usepackage[T2A]{fontenc}
\usepackage[cp1251]{inputenc}

\newtheorem{theorem}{Theorem}[section]
\newtheorem{remark}[theorem]{Remark}
\newtheorem{proposition}[theorem]{Proposition}
\newtheorem{lemma}[theorem]{Lemma}
\newtheorem{corollary}[theorem]{Corollary}

\definecolor{darkred}{rgb}{0.9,0.1,0.1}

\begin{document}
\title{Large Deviations for Markov jump processes in periodic and locally periodic environments}
\author{A. Piatnitski$^{\rm a,b}$, \ S. Pirogov$^{\rm b}$ and E. Zhizhina$^{\rm b}$\\[3mm]
$^{\rm a}$ \small The Arctic University of Norway, UiT, campus Narvik, Norway\\[1mm]
$^{\rm b}$ \small Institute for Information Transmission Problems of RAS, Moscow, Russia
}

\date{}

\maketitle

\newcommand{\eps}{\varepsilon}

\medskip

\begin{abstract}
The paper deals with
a family
of jump Markov process defined in a medium with a periodic or locally periodic microstructure.
We assume that the generator of the process is a zero order convolution type operator
with rapidly oscillating locally periodic coefficient and, under natural ellipticity and
localization conditions, show that the family satisfies the large deviation principle
in the path space equipped with Skorokhod topology.  The corresponding rate function is defined
in terms of a family of auxiliary periodic spectral problems.
\end{abstract}

\newtheorem{Theorem}{Theorem}
\newtheorem{Proposition}{Proposition}

\section{Introduction}
The goal of this work is to show that for a family of jump Markov process defined in a $d$-dimensional medium with a (locally)
periodic microstructure the large deviation principle holds. We assume that the generators of these processes
are of the form
\begin{equation}\label{generator_intro}
{\textstyle
A^\varepsilon u (x) = \frac{1}{\varepsilon^{d+1}} \int_{\mathbb{R}^d} a(\frac{x-y}{\varepsilon}) \Lambda^\eps(x,y) (u(y) - u(x)) dy,  }
\end{equation}
where $\eps$ is a small positive parameter that characterizes the microscopic length scale, $a(\cdot)$ is
a non-negative integrable convolution kernel that decays super exponentially at infinity, and a positive
bounded function $\Lambda^\eps$ represents the local characteristics of the medium.  We consider both the case
of a periodic function $\Lambda^\eps$, and the case of a locally periodic one.  In the former case,
$\Lambda^\eps(x,y)=\Lambda\big(\frac x\eps,\frac y\eps\big)$, where $\Lambda(\xi,\eta)$ is a periodic
function in $\mathbb R^{2d}$. In the latter case, $\Lambda^\eps(x,y)=\Lambda\big(x,y,\frac x\eps,\frac y\eps\big)$, where $\Lambda(x,y,\xi,\eta)$ is  periodic in $\xi$ and $\eta$.

Previously,  the large deviation principle for trajectories of a diffusion process with a small diffusion coefficient has been justified
in \cite{WF_70}, \cite{FW}.
It was shown that the large deviation principle holds in the space of continuous functions and that the corresponding rate function is defined as an integral along  the curve of an appropriate Lagrangian. The Lagrangian is explicitly given in terms of the coefficients of the process generator.

 Large deviation problem for a diffusion in environments with a periodic microstructure was studied for the first time
 in \cite{Bal91}, where a pure diffusion without drift has been considered. The case of a small diffusion with a drift
 in locally periodic media was studied in \cite{FrSow}.   Here the Lagrangian is defined in terms of an auxiliary PDE
 problem on the torus.

Large deviation problems for jump processes with independent increments have been investigated in \cite{Bor67},
\cite{M1993}, \cite{M2016}, \cite{Pukh94} and other works. In \cite{Bor67} the author considered the one-dimensional case.  The LDP was
obtained in the Skorokhod space with a weak topology under the Cramer condition on the convolution kernel. These results were improved
in \cite{M1993}, \cite{M2016}, where the LDP was proved in the Skorokhod space with strong topology and the topology
of uniform convergence.  In the multidimensional case similar results were obtained in \cite{Pukh94}.

A number of interesting results on large deviations for Markov processes that combine a diffusive 
behaviour and many small jumps can be found in \cite{We86}.

The monograph \cite{FeKu2005} focuses on LDP for rather general classes of 
Markov processes in metric spaces. 
The approaches developed in this book rely on exponential tightness,
convergence of nonlinear contraction semigroups and theory of viscosity
solutions of nonlinear equations. In particular, this allows to consider the case 
of processes whose rate function might be finite for sample
paths with discontinuities.

To our best knowledge, large deviation problems for jump Markov processes in environments with
a periodic microstructure have not been studied in the existing literature.

In the present paper we consider a family of jump Makov processes $\xi^\eps(t)$, $0\leq t\leq T$,  with the generator defined in \eqref{generator_intro}.  Under the assumptions that the convolution kernel $a(\cdot)$ is integrable and decays super exponentially
at infinity, and that   the function $\Lambda^\eps$ is strictly positive, bounded and has a periodic or locally periodic
microstructure we prove that the family       $\{\xi^\eps(t)\}$ satisfies the large deviation principle in the Skorokhod
space $\mathbf{D}([0,T];\mathbb R^d)$ equipped with the strong topology. The corresponding rate function is good, it is  finite only for absolutely continuous functions and is given by
$$
I(\gamma(\cdot))=\int_0^T L\big(\gamma(t),\dot\gamma(t)\big)dt,
$$
where the Lagrangian $L(x,\zeta)$ is convex and has a  super linear growth as a function of $\zeta$ while in $x$
it is continuous. This Lagrangian is constructed in terms of a family of auxiliary periodic spectral problems
for operators which are derived from the generator of the process by the exponential transformation.

It turns out that the said Lagrangian need not be strictly convex. This is one of the interesting features of the studied problem.  More precisely, the Lagrangian is strictly convex in the vicinity of infinity. However,  its restriction
on some segments going from the origin can be a linear function. In Section \ref{s_perio} we provide an example of
such a Lagrangian.

The paper is organized as follows. In Section \ref{s_setup} we introduce the studied  family of jump Markov
 processes and provide all our assumptions.

 In Section \ref{s_indep} we recall some of the existing large deviation results for jump process with independent increments.

The case of purely periodic environment is considered in Section \ref{s_perio}. First we introduce a family of
auxiliary operators with periodic coefficients, consider the corresponding spectral problems on the torus, and study the structure of their spectrum. Then we define the Hamiltonian and the Lagrangian that are required for formulating the large deviation results, and investigate their properties.   In the last part of this section
we formulate and  prove the large deviation theorems, first for the distribution of the process in $\mathbb R^d$
at a fixed time, and then in the path space.

Section \ref{s_slowly} deals with the media that do not depend on fast variables. Here we combine the results obtained for the processes with independent increments and perturbation theory arguments. Although this idea
is very natural and not new, its realization requires a number of quite delicate technical statements.

Finally, in the last section we consider the generic case of locally periodic media.

%


\section{Problem setup}\label{s_setup}

We consider a family of continuous time jump Markov processes $\xi_{x_0}^\varepsilon(t)$ in environments with locally periodic microstructure that depend on a small parameter $\eps>0$; the subindex $x_0$ indicates the starting point:  $\xi_{x_0}^\eps(0)=x_0$.
The generator of this process has the form
\begin{equation}\label{A}
A^\varepsilon u (x) = \frac{1}{\varepsilon^{d+1}} \int_{\mathbb{R}^d} a(\frac{x-y}{\varepsilon}) \Lambda(x,y, \frac{x}{\varepsilon},  \frac{y}{\varepsilon}) (u(y) - u(x)) dy,
\end{equation}
$u \in L^2(\mathbb{R}^d)$. We call $x,y$ slow variables and $\frac{x}{\varepsilon}, \frac{y}{\varepsilon}$ fast variables.

Our goal is to show that, under proper ellipticity and exponential moment conditions, the large deviation principle
holds for this family of Markov processes.  In this section we introduce these conditions.

\bigskip
For the function $a(z)$ we assume that
\begin{equation}\label{M1}
a(z) \in  L^{1}(\mathbb R^d) \cap L^\infty(\mathbb R^d), \; \; a(z) \ge 0,\quad
\| a \|_{L^1(\mathbb R^d)}  = \int_{\mathbb R^d} a(z) \, dz = 1,
\end{equation}
and the convolution kernel $a(z)$ satisfies the following upper bound with some  ${p} > 1$, $k>0,\, \mathtt{C}>0$:
\begin{equation}\label{lt}
0 \le a(z) \le \mathtt{C} e^{- k |z|^{p}}.
\end{equation}
The latter condition
implies in particular  that all exponential moments are bounded.\\
We assume furthermore that  for all $\alpha$ from the unit sphere  $S^{d-1}$ we have
\begin{equation}\label{M2-bis}
\int_{\Pi_\alpha} a (z) dz >0 \quad   \mbox{with } \; \Pi_\alpha  = \{ z \in \mathbb R^d, \ z \cdot \alpha >0 \}.
\end{equation}
Observe, that the integral  $\int_{\Pi_\alpha} a (z) dz$ is a continuous function of $\alpha \in S^{d-1}$ and,
therefore,
\begin{equation}\label{M2_conseq}
 \min\limits_{\alpha \in S^{d-1}} \int_{\Pi_\alpha} a (z) dz \geq C_0
\end{equation}
for some $C_0>0$.

The function $\Lambda(x,y, \xi, \eta) $ describes the locally periodic environment. We assume that the function $\Lambda$ is periodic in $\xi$ and $\eta$,
\begin{equation}\label{cond_perio}
\begin{array}{l}
\displaystyle
\Lambda(x,y, \xi+j', \eta+j'')=\Lambda(x,y, \xi, \eta)\quad\hbox{for all }j',\,j''\in\mathbb Z^d\\
\hfill\hbox{and for all }x,\,y,\,\xi,\,\eta\in\mathbb R^d,
\end{array}
\end{equation}
and that
\begin{equation}\label{cond_conti}
\begin{array}{r}
\Lambda(x,y, \xi, \eta) \quad \hbox{is uniformly contunuous in $x$ and $y$}\\
\hbox{and measurable in $(\xi,\eta)$
for each $x$ and $y$.}
\end{array}
\end{equation}
We assume furthermore that $\Lambda$ is bounded from above and from below:
\begin{equation}\label{lm}
0 < \Lambda^- \le \Lambda(x,y, \xi, \eta) \le \Lambda^+ < \infty.
\end{equation}

\section{Processes with independent increments}
\label{s_indep}

We start with the case of constant $\Lambda$:  $\Lambda^\eps(x,y)\equiv \Lambda$. In this case $\xi_x^\eps(\cdot)$ is a continuous time process with independent increments, or equivalently a compound Poisson process. The results on large deviations under condition \eqref{lt} are well known, see e.g. \cite{Bor67}. In \cite{LS, M1993, M2016} the authors considered  a wider class of the compound Poisson processes that have exponential moments only in a neighborhood of zero.
Let us shortly repeat the construction of the rate function and the Lagrangian for this process.

In this section the dependence of $\Lambda$ is indicated explicitly,  $\xi_{x,\Lambda}^{\varepsilon}(t)$ stands for a continuous time process with independent increments whose generator is defined by
\begin{equation}\label{A-1eps}
A_\Lambda^\varepsilon u (x) = \frac{\Lambda}{\varepsilon^{d+1}} \int_{\mathbb{R}^d} a(\frac{x-y}{\varepsilon})  (u(y) - u(x)) dy, \; u \in L^2(\mathbb{R}^d).
\end{equation}
To apply the G\"{a}rtner-Ellis theorem we consider the family of probability measures
$\mu^{\varepsilon, x}_{\Lambda,t}$ in $\mathbb R^d$ defined as the law of the random variables $\xi_{x,\Lambda}^{\varepsilon}(t)$. In what follows we assume without loss of generality that $x=0$
and drop the index $x$.
We also consider the process $\xi_\Lambda(t)$ generated by
\begin{equation}\label{A-1}
A_\Lambda u (x) =  \Lambda\int_{\mathbb{R}^d} a(x-y) (u(y) - u(x)) dy, \; u \in L^2(\mathbb{R}^d).
\end{equation}
It is worth to notice that
$$
\xi_\Lambda^\varepsilon(t) = \varepsilon \, \xi_\Lambda \big( \frac{t}{\varepsilon} \big).
$$
We have
$$
{\mathbb{E}} e^{\lambda \xi_\Lambda(T)}  = e^{T H_\Lambda(\lambda)}
$$
with
\begin{equation}\label{spec}
H_\Lambda(\lambda) = \Lambda\Big( \int a(z) e^{-\lambda z} dz -1\Big)=\Lambda H(\lambda).
\end{equation}
Representation \eqref{A-1eps} for the generator $A_\Lambda^\varepsilon$ yields
\begin{equation}\label{E-eps}
{\mathbb{E}} e^{\frac{\lambda}{\varepsilon} \xi_\Lambda^{\varepsilon}(t)}  = e^{\frac{t}{\varepsilon} H_\Lambda(\lambda)},
\end{equation}
Thus, we get
\begin{equation}\label{H}
 \lim_{\varepsilon \to 0} \varepsilon \ln {\mathbb{E}} \, e^{\frac{\lambda}{\varepsilon} \xi_\Lambda^{\varepsilon}(t)}= t H_\Lambda(\lambda)=t\Lambda H(\lambda).
\end{equation}
Relation \eqref{spec} readily implies that the function $H_\Lambda(\lambda)$ is a smooth, strictly convex and of super-linear growth at infinity. Denote by $L(\zeta)$ the  Legendre transform of $ H(\lambda)$:
\begin{equation}\label{Leg1}
L(\zeta) =   \sup\limits_\lambda \big\{ \lambda \zeta -  H(\lambda) \big\}.
\end{equation}
Then the function $t\Lambda L(\frac{\zeta}{t\Lambda})$ is the Legendre transform of $ t H_\Lambda(\lambda)$:
$$
\sup\limits_\lambda \{ \lambda \zeta - t H_\Lambda(\lambda)  \} = t\Lambda \, \sup\limits_\lambda \Big\{ \lambda \frac{\zeta}{t\Lambda} -  H(\lambda) \Big\} = t\Lambda \, L\Big(\frac{\zeta}{t\Lambda}\Big).
$$
The function $L(\zeta)$ is  
non-negative, strictly convex and finite for any $\zeta\in\mathbb R^d$.
Consequently, by the G\"{a}rtner-Ellis theorem LDP holds in this case:


1) for every closed set $C \subset \mathbb{R}^d$
\begin{equation}\label{ULDB}
\limsup_{\varepsilon \to 0} \varepsilon \ln \mathbb{P}(\xi_\Lambda^\varepsilon (t) \in  C ) \le - \inf\limits_{\zeta \in  C} \Big[ t\Lambda \, L\Big(\frac{\zeta}{t\Lambda} \Big)\Big];
\end{equation}

2) for every open set $O \subset \mathbb{R}^d$
\begin{equation}\label{LLDB}
\liminf_{\varepsilon \to 0} \varepsilon \ln \mathbb{P}(\xi_\Lambda^\varepsilon (t) \in  O ) \ge
 - \inf_{\zeta \in  O} \Big[t\Lambda\, L\Big(\frac{\zeta}{t\Lambda} \Big)\Big].
\end{equation}


\begin{remark}
The case when $a(z)$ is a symmetric kernel, i.e. $a(-z) = a(z)$, and $\Lambda(\cdot) \equiv 1$ has been studied in \cite{GKPZ}.
In particular, the large deviation result for the density $v(x,t)$ of the transition probability $\Pr(\xi(t) = x| \, \xi(0)=0)$ has been proved with the rate function $\Phi(\zeta), \, x = \zeta t(1+o(1)), \; t\to\infty$, see Theorems 3.4 and 3.8, \cite{GKPZ}.
The rate function $\Phi(\zeta)$  possesses the following properties: \\
$\Phi(0)=0$,  $\Phi(\zeta)>0$ for $\zeta \neq 0$, $\Phi$ is a convex function, and
\begin{equation}\label{Phi0}
\Phi(\zeta) = \frac12 (\sigma^{-1}\zeta, \zeta) (1+ o(1)), \quad \mbox{ as } \; |\zeta| \to 0,
\end{equation}
where $\sigma$ is the covariance matrix, $\sigma_{i j} = \int_{\mathbb R^d} x_i x_j a(x) dx$.

If the function $a(x)$ satisfies a two-sided estimate
\begin{equation*}
C_2 e^{-b |x|^p} \le a(x) \le C_1 e^{-b |x|^p}, \quad p> 1,
\end{equation*}
then the following asymptotics for the rate function $\Phi(\zeta)$ holds:  
\begin{equation}\label{Phi}
\Phi (\zeta) ={\textstyle \frac p{p-1} }\big(b(p-1)\big)^{1/p} |\zeta| (\ln |\zeta|)^{\frac{p-1}{p}} \, (1+o(1)), \qquad
\hbox{as } |\zeta| \to \infty .
\end{equation}
\end{remark}

Relation \eqref{Phi} has an important consequence that will be used in the following sections.
Namely, under condition \eqref{lt}, there exists a constant $c_0=c_0(C,p,d)$ such that for all sufficiently large
$\zeta$ the inequality
\begin{equation}\label{Phi_oneside}
\Phi (\zeta) \geq c_0 |\zeta| (\ln |\zeta|)^{\frac{p-1}{p}}
\end{equation}
holds true.

\medskip
Finally, we turn to the sample path large deviations results. Denote by ${\bf P}^\varepsilon$ the distribution of paths of the process $\xi_\Lambda^\varepsilon(t), \, 0 \le t \le T,$ in the space ${\bf D}([0,T];\mathbb R^d)$.
This space is equipped with the metric
$$
\mathrm{dist}(f,g)=\inf\limits_{\pi(\cdot)}\max\Big\{ {\textstyle
\sup\limits_{0\leq s<t\leq T}\big|\log\big(\frac{\pi(t)-\pi(s)}{t-s}\big)\big|,\sup\limits_{0\leq t\leq T}|f(t)-
g(\pi(t))| } \Big\},
$$
 where the infimum is taken over all continuous strictly monotone functions $\pi$ such that
$\pi(0)=0$ and $\pi(T)=T$.  In what follows this set of functions is denoted by $\mathcal{K}$, and
$\ell(\pi)=\sup\limits_{0\leq s<t\leq T}\big|\log\big(\frac{\pi(t)-\pi(s)}{t-s}\big)\big|$.

In the case of the studied process with independent increments the large deviation principle (LDP) is valid for  the family of probability measures $\{{\bf P}^\varepsilon \}$ in the Skorokhod space equipped with topology generated by the above introduced metric,
the rate function being given by
$$
I_\Lambda(\gamma(\cdot)) = \left\{\begin{array}{ll}
\displaystyle
\int_0^T \Lambda {\textstyle L\big(\frac1\Lambda\dot\gamma(t)\big) dt},&\hbox{if }\gamma(\cdot)\ \hbox{is absolutely continuous,}\\[2.4mm]
+\infty, &\hbox{otherwise},
\end{array}
\right.
$$
with $L(\cdot)$ defined in \eqref{Leg1}.
This means that
\begin{equation}\label{ULDB_func}
\limsup_{\varepsilon \to 0} \varepsilon \ln \mathbf{P}^\eps(C ) \le - \inf\limits_{\gamma \in  C} \big[ I_\Lambda(\gamma)\big]
\end{equation}
for every closed set $C$ in ${\bf D}([0,T];\mathbb R^d)$, and
\begin{equation}\label{LLDB_func}
\liminf_{\varepsilon \to 0} \varepsilon \ln \mathbf{P}^\eps(  O ) \ge
 - \inf_{\gamma \in  O} \big[I_\Lambda(\gamma)\big]
\end{equation}
for every open set $O$ in ${\bf D}([0,T];\mathbb R^d)$.

As a consequence, for a small neighbourhood $U$ of a curve $\gamma$ we have
\begin{equation}\label{PathPDP}
\varepsilon \ln {\bf P}^\varepsilon(U) \sim - 
I_\Lambda(\gamma), \quad \mbox{as } \, \varepsilon \to 0.
\end{equation}
In the one-dimensional case   this result was proved, under slightly weaker assumptions,  by A. Mogulskii in \cite{M1993, M2016}, and then in multidimensional case by A. Pukhalskii in \cite{Pukh94}.

\section{Environment with periodic microstructure $ \Lambda(\frac{x}{\varepsilon},  \frac{y}{\varepsilon})$}
\label{s_perio}

In this section we consider the process with generator given by  \eqref{A} with $\Lambda = \Lambda(\frac{x}{\varepsilon}, \frac{y}{\varepsilon})$, where $\Lambda(\eta,\zeta)$ is a measurable periodic function satisfying the lower and upper bounds in  \eqref{lm}.

\subsection{Skewed generator}

Consider an operator
\begin{equation}\label{A_0}
A_0 u (x) = \int_{\mathbb{R}^d} a(x-y) \Lambda(x,y) u(y) dy - \int_{\mathbb{R}^d} a(x-y) \Lambda(x,y) dy \, u(x),
\end{equation}
where $\Lambda(x,y)$ is a periodic function satisfying bound \eqref{lm}, and $ u \in L^2(\mathbb{R}^d)$. Denote by  $S(t)=e^{t A_0}$ the Markov semigroup with generator $A_0$, and let  $\xi_x(t)$ be
the corresponding continuous time jump Markov process starting at $x$. Then
\begin{equation}\label{St}
(S(t) f)(x) = e^{t A_0} f(x)  = {\mathbb{E}} f(\xi_x(t)).
\end{equation}

\begin{lemma}
For any $\lambda \in \mathbb{R}^d$ and $x \in \mathbb{R}^d$
\begin{equation}\label{exp}
\mathbb{E} e^{\lambda \xi_x(t)} = e^{\lambda x} e^{t A_\lambda} 1,
\end{equation}
where $A_\lambda$ is the operator acting in the space of periodic functions $L^2(\mathbb{T}^d)$ and defined by
\begin{equation}\label{A_l}
A_\lambda v (x) = \int_{\mathbb{R}^d} a(x-y) \Lambda(x,y) e^{\lambda(y-x)}v(y) dy - \int_{\mathbb{R}^d} a(x-y) \Lambda(x,y) dy \, v(x). \end{equation}
\end{lemma}

\begin{proof}
Substitute  $f(z) = e^{\lambda z}$ in \eqref{St} and denote $u(x,t) = \mathbb{E} e^{\lambda \xi_x(t)}$. Under our
standing assumptions on $a(\cdot)$ the function $u(\cdot)$ is well defined. Indeed,
denoting
$$
\tilde p^t_x(y) = e^{-t \Lambda^-} \delta_x(y) +  e^{-t \Lambda^-} \sum_{n=1}^{\infty} \frac{(\Lambda^+)^n \, t^n}{n!}  a^{\star n}(x-y)
$$
with $\Lambda^-$ and  $\Lambda^+$ defined in \eqref{lm} we have
$$
p^t_x (y) \le \tilde p^t_x (y),
$$
where $p^t_x(\cdot) =  e^{t A_0} \delta_x (\cdot)$ is the distribution of the process $\xi_x(t)$.
Considering    \eqref{Phi_oneside},       
 in the same way as in
\cite{ZhP}, one can show that $\tilde p^t_x(y)$ does not exceed  $ \ e^{- c \frac{|x-y|}t\big( \ln \frac{|x-y|}t \big)^{\frac{p-1}{p}}}$
for some $c>0$ and for all $y$ such that $|x-y|\geq (1\vee t)$. Consequently, the integral
$ \int_{\mathbb{R}^d} e^{\lambda y} p^t_x (y) dy$ converges for any $t>0$ and $\lambda\in\mathbb R^d$,
and the function
\begin{equation}\label{St-0}
u(x,t) = \int_{\mathbb{R}^d} e^{\lambda y} p^t_x (y) dy,
\end{equation}
is well defined. Moreover,  due to periodicity of $\Lambda(x,y)$,
\begin{equation}\label{St-1}
v(x,t) = e^{- \lambda x} u(x,t)=\int_{\mathbb{R}^d} e^{\lambda (y-x)} p^t_x (y) dy= B_\lambda^{-1} e^{t A_0} B_\lambda \, 1
\end{equation}
is a periodic function of $x$, i.e. $v(\cdot,t) \in  L^2(\mathbb{T}^d)$ for any $t>0$;
here  $B_\lambda g (x) = e^{\lambda x} g(x)$. In fact, under our assumptions $v(\cdot,t)\in L^\infty(\mathbb T^d)$.
%
%
%
Since
$A_\lambda = B_\lambda^{-1} A_0 B_\lambda$, where $A_\lambda$ is defined by \eqref{A_l}, we have
$B_\lambda^{-1} e^{t A_0} B_\lambda = e^{t  A_\lambda}$.
This yields \eqref{exp}.
\end{proof}

Consequently, for any $t>0$, we have
\begin{equation}\label{St-3}
\lim_{\varepsilon \to 0} \varepsilon \, \ln \mathbb{E} e^{\frac{\lambda}{\varepsilon} \xi_0^\varepsilon(t)} = \lim_{\varepsilon \to 0} \varepsilon \, \ln\big(  [e^{t A^\varepsilon_{\lambda/\varepsilon}}1] (0)\big)
=\lim_{s\to+\infty}{\textstyle\frac1s} \ln\big(  [e^{ts A_{\lambda}}1] (0)\big),
\end{equation}
where $ A^\varepsilon_{\lambda/\varepsilon}=B^{-1}_{\lambda/\eps} A^\eps B_{\lambda/\eps}$.
It is straightforward to check that for any $\lambda\in\mathbb R^d$ the skewed operator $A_\lambda$ is bounded in $L^2(\mathbb T^d)$.
Denote by $\sigma(A_\lambda)$ the spectrum of this operator in $L^2(\mathbb T^d)$,
and by $\mathtt{s}(A_\lambda)$ the maximum of the real parts of the elements of $\sigma(A_\lambda)$.
In the next subsection we will show that the limit on the right-hand side of \eqref{St-3} exists and is equal to
$\mathtt{s}(A_\lambda)$ multiplied by $t$. Our goal is to study the properties  of $\mathtt{s}(A_\lambda)$
as a function of $\lambda$.


\subsection{The spectral properties of the operator  $A_\lambda$ }
\label{ss_prop_alambda}

The operator $A_\lambda$ defined by \eqref{A_l}
has a continuous spectrum
$$
\sigma_{\rm cont} =  [-g_{\rm max}, -g_{\rm min}] := \mathrm{Im}\{ -G(x)\}, \quad  x \in \mathbb{T}^d,
$$
if the function
$$
G(x) = \int_{\mathbb{R}^d} a(x-y) \Lambda(x,y) dy
$$
is not a constant. Letting
$$
 g_{\rm max} = \max_{x \in \mathbb{T}^d} G(x), \qquad
g_{\rm min} = \min_{x \in \mathbb{T}^d} G(x),
$$
we have $ 0<g_{\rm min}\leq g_{\rm max}<\infty$. The continuous spectrum, if exists, does not depend on $\lambda$. In addition, depending on the value of $\lambda$, $A_\lambda$ might have
a discrete spectrum $\sigma_{\rm disc}(\lambda)$.

Adding to the both sides of the spectral problem $A_\lambda v=\theta v$ the constant $g_{\rm max}$ we
obtain an equivalent spectral problem that reads $(A_\lambda+g_{\rm max}) v=(\theta+g_{\rm max}) v$.
We denote the new spectral parameter $(\theta+g_{\rm max})$ by $\vartheta$.
The operator on the left-hand side of the latter spectral problem is positive, its essential spectrum coincides with its
continuous spectrum and is equal to the real interval $[0, g_{\rm max}-g_{\rm min}]$.
According to \cite[Theorem 1]{EdPoSt72}  there are only two options. Namely, either for
any $\vartheta\in\sigma(A_\lambda+g_{\rm max})$ we have $|\vartheta|\leq  g_{\rm max}-g_{\rm min}$, or
there exists a real positive eigenvalue $\vartheta(\lambda)$ of $A_\lambda+g_{\rm max}$ such that
$\vartheta(\lambda)>|\tilde\vartheta|$ for any
$\tilde\vartheta\in\sigma(A_\lambda+g_{\rm max})\setminus \vartheta(\lambda)$. In particular, in the latter
case, $\vartheta(\lambda)>g_{\rm max}-g_{\rm min}$.  Furthermore, there is a positive eigenfunction $u_\lambda$
that  corresponds to  $\vartheta(\lambda)$.

As a consequence, either the element of $\sigma(A_\lambda)$ with the largest real part coincides with $-g_{\rm min}$, or it is equal to $\vartheta(\lambda)-g_{\rm max}$. The latter case takes place if and only if $\theta(\lambda):=\vartheta(\lambda)-g_{\rm max}>-g_{\rm min}$, in this case the real part of $\tilde\theta$
is less than $\theta(\lambda)$ for any $\tilde\theta\in\sigma(A_\lambda)\setminus\theta(\lambda)$.
The set of $\lambda\in\mathbb R^d$ such that $\theta(\lambda)>-g_{\rm min}$ is denoted by $\Gamma $, and $\theta(\lambda)$ is called the principal eigenvalue of  $A_\lambda$.
\begin{remark}
Notice that $\theta(0) = 0$, i.e. $\theta(0) > -g_{\rm min}$. Furthermore, $\theta(\lambda) \to \infty$ as $|\lambda| \to \infty$. Thus,  $0 \in \Gamma$, and $\mathbb R^d\setminus\Gamma$ is a bounded set.
\end{remark}

Assume that $\lambda\in\Gamma$. The spectral problem for $A_\lambda$ reads
\begin{equation}\label{specA}
\int\limits_{\mathbb{R}^d}\!\! a(x-y) \Lambda(x,y) e^{\lambda(y-x)}u_\lambda(y) dy
-\! \int\limits_{\mathbb{R}^d}\!\! a(x-y) \Lambda(x,y) dy \, u_\lambda(x)
= \theta(\lambda) u_\lambda (x),
\end{equation}
where $u_\lambda(x)$ is the principle eigenfunction. Denote by $u_\lambda^\star (x)$ the principle eigenfunction of the adjoint operator $A_\lambda^\star$. For $\theta(\lambda) > -g_{\rm min}$, the spectral problem \eqref{specA} is equivalent to the following problem
$$
D_\lambda u(x) = \big(G(x)+  \theta(\lambda) \big)^{-1} \int_{\mathbb{R}^d} a(x-y) \Lambda(x,y) e^{\lambda(y-x)}u_\lambda(y) dy =  u_\lambda (x)
$$
for the compact positive operator $D_\lambda$ in $L^2(\mathbb T^d)$. Since $\theta(\lambda)$ is an eigenvalue for $A_\lambda$, $1$ is an eigenvalue for $D_\lambda$.

For an arbitrary $N\in\mathbb Z^+$ denote by $\beta_N(x,y)$ the kernel of the operator $D^N_\lambda$:
\begin{equation}\label{power_of_oper}
D^N_\lambda v(x)=\int_{\mathbb T^d} \beta_N(x,y)v(y)\,dy.
\end{equation}
Then there exist $N\in\mathbb Z^+$ and constants $\beta^->0$ and $\beta^+$ such that
\begin{equation}\label{kernel_of_power}
\beta^-\leq \beta_N(x,y)\leq\beta^+ \quad\hbox{for all }x,\,y\in\mathbb T^d.
 \end{equation}
The lower bound was proved, for instance, in \cite[Lemma 4.1]{PZh2019}. The upper bound is evident.

Recalling that $u_\lambda$ is positive,   by the Krein-Rutman theorem, see e.g. \cite[\S 6, Proposition $\beta'$]{KLS},  $1$ is the principal
eigenvalue of  $D_\lambda$, and this eigenvalue is simple. Then $\theta(\lambda)$ is also simple.

From \eqref{kernel_of_power} it readily follows that both for
 $u_\lambda(x)$ and for $u_\lambda^\star(x)$ the following bounds hold
 \begin{equation}\label{est_eiggg}
c^-\leq u_\lambda(x)\leq c^+ \quad\hbox{and}\quad c^-\leq u^\star_\lambda(x)\leq c^+ \quad
\hbox{for all }x \in\mathbb T^d
 \end{equation}
 for some constants $c^->0$ and $c^+$.
 In what follows we assume the following normalization conditions to hold:
\begin{equation}\label{normalization}
\int_{\mathbb{T}^d} u_\lambda(x) dx =1, \qquad \int_{\mathbb{T}^d} u_\lambda(x) u_\lambda^\star (x) dx = 1.
\end{equation}

We now turn to relation \eqref{St-3}.
\begin{lemma}\label{l_growth_a_lambda}
The limit on the right-hand side of \eqref{St-3}  exists and
is equal to $t \,\mathtt{s}(A_\lambda)$.
\end{lemma}
\begin{proof}
According to \cite[Corollary IV.2.4]{EnNa99}, the following relation holds:
$$
 \lim_{s\to+\infty}{\textstyle\frac1s} \ln\big\| e^{ts A_{\lambda}}\big\|_{\mathcal{L}(L^2(\mathbb T^d),L^2(\mathbb T^d))}=t \mathtt{s}(A_\lambda).
$$
This readily yields an upper bound
$$
\limsup_{s\to+\infty}{\textstyle\frac1s} \ln\big(  [e^{ts A_{\lambda}}1] (0)\big)\leq t \mathtt{s}(A_\lambda).
$$

To obtain the lower bound we consider separately the cases $\lambda\in\Gamma$ and
$\lambda\in\mathbb R^d\setminus\Gamma$. If $\lambda\in\Gamma$, then
$\mathtt{s}(A_\lambda)=\theta(\lambda)$, and the inequality
$$
\liminf_{s\to+\infty}{\textstyle\frac1s} \ln\big(  [e^{ts A_{\lambda}}1] (0)\big)\geq t \mathtt{s}(A_\lambda)
$$
follows from the facts that $u_\lambda$ is positive  and that $e^{ts A_{\lambda}}$ is a positive operator.

If $\lambda\in\mathbb R^d\setminus\Gamma$ then $\mathtt{s}(A_\lambda)=-g_{\rm min}$. Consider an
auxiliary semigroup with the generator $(\mathcal{G}u)(x)=-G(x)u(x)$. It is straightforward to check that
$$
\lim_{s\to+\infty}{\textstyle\frac1s} \ln\big(  [e^{ts \mathcal{G}}1] (0)\big)= -t g_{\rm min}.
$$
Since the operator $A_\lambda-\mathcal{G}
=(A_\lambda+g_{\rm max})-(\mathcal{G}+g_{\rm max})$ is positive, the operator $e^{st A_\lambda}-e^{st \mathcal{G}}$
is also positive, and we conclude that
$$
\liminf_{s\to+\infty}{\textstyle\frac1s} \ln\big(  [e^{ts A_\lambda}1] (0)\big)\geq -t g_{\rm min}.
$$
This completes the proof.
\end{proof}
Our next statement describes the behaviour of $\theta(\lambda)$ at infinity.
\begin{lemma}\label{l_exp_growth}
  There exists $R_0>0$ such that $\mathtt{s}(A_\lambda)>-g_{\rm min}$  for all $\lambda$ with
  $|\lambda|\geq R_0$. Moreover, there exist constants $c_e>0$, $c_a>0$ and $C_s$ such that
  $$
  \theta(\lambda)\geq c_a e^{c_e|\lambda|}-C_s
  $$
  for all $\lambda\in\{\lambda\in\mathbb R^d\,:\,|\lambda|\geq R_0\}$.
\end{lemma}
\begin{proof}
It follows from \eqref{M1} and \eqref{M2_conseq} that for any $\alpha\in S^{d-1}$ there exist a ball $Q^\alpha\subset \Pi_\alpha$ such that
$$
c^\alpha_1:=\mathrm{dist}(Q^\alpha,\partial\Pi_\alpha)>0\quad
\hbox{and } \ c^\alpha_2:=\int_{Q^\alpha}a(-z)\,dz>0.
$$
Then,  for $\lambda=r\alpha$ with $r>0$ we have
$$
\int_{\mathbb R^d}a(x-y)e^{\lambda\cdot(y-x)}\Lambda(x,y)dy\geq  \Lambda^-c_2^\alpha e^{c_1^\alpha r}=
\Lambda^- c_2^\alpha e^{c_1^\alpha |\lambda|}
$$
By the continuity argument,
$$
\int_{\mathbb R^d}a(x-y)e^{\lambda\cdot(y-x)}\Lambda(x,y)dy\geq
\Lambda^- c_2^\alpha e^{\frac12c_1^\alpha |\lambda|}
$$
if $\frac \lambda{|\lambda|}$ belongs to a sufficiently small neighbourhood of $\alpha$. Due to the compactness
of $S^{d-1}$ this implies that for some $c_a>0$ and $c_e>0$ the inequality
$$
\int_{\mathbb R^d}a(x-y)e^{\lambda\cdot(y-x)}\Lambda(x,y)dy\geq
c_a e^{c_e |\lambda|}
$$
holds for all $\lambda\in \mathbb R^d$. Therefore,
$[(A_\lambda+g_{\rm max})1](x)\geq c_a e^{c_e |\lambda|}$. Since the operator $A_\lambda+g_{\rm max}$
is positive, this yields
$[(A_\lambda+g_{\rm max})^n1](x)\geq c^n_a e^{nc_e |\lambda|}$ for any $n\in\mathbb Z^+$, and we conclude
that $\vartheta(\lambda)\geq c_a e^{c_e |\lambda|}$, and
$\theta(\lambda)\geq c_a e^{c_e |\lambda|}-g_{\rm max}$.
\end{proof}

\subsection{Strict convexity of the principal eigenvalue $\theta(\lambda)$ of the operator  $A_\lambda$ }
\label{ss_strict_conv}

\begin{theorem}\label{SP}
The function $\theta(\lambda)$  is strictly convex on $\Gamma$, i.e.
$\frac{\partial^2 \theta}{\partial \lambda_i \partial \lambda_j} (\lambda)$
is a positive definite matrix for all $\lambda \in \Gamma$.
\end{theorem}

\begin{proof}
We are going to show that the matrix $\nabla \nabla \theta (\lambda_0)$ coincides with an effective diffusion matrix for a family of  convolution type operators with periodic coefficients.

Let us start with the case $\lambda_0 = 0$. Then $\theta(0)=0$, and the principal eigenfunction $u_0(x) \equiv 1$.
Differentiating equality \eqref{specA} in $\lambda_i$, $i=1,\ldots,d$, yields
\begin{equation}\label{specA-dif1}
\begin{array}{l}
\displaystyle
\int_{\mathbb{R}^d} a(x-y) \Lambda(x,y) (y_i - x_i) e^{\lambda(y-x)}u_\lambda(y) dy \\
\displaystyle
+  \int_{\mathbb{R}^d} a(x-y) \Lambda(x,y) e^{\lambda(y-x)} \partial_{\lambda_i}  u_\lambda(y) dy \\[4mm]
\displaystyle
 - \int_{\mathbb{R}^d}\!\! a(x-y) \Lambda(x,y) dy \, \partial_{\lambda_i} u_\lambda(x)
  = \big(\partial_{\lambda_i} \theta(\lambda) \big) u_\lambda (x) + \theta(\lambda)\big( \partial_{\lambda_i} u_\lambda(x) \big).
\end{array}
\end{equation}
Relation  \eqref{specA-dif1} can be rearranged as follows:
\begin{equation}\label{specA-dif1-bis}
\begin{array}{l}
\displaystyle
\int_{\mathbb{R}^d} a(x-y) \Lambda(x,y) e^{\lambda(y-x)} \partial_{\lambda_i}  u_\lambda(y) dy \\
\displaystyle
 - \int_{\mathbb{R}^d} a(x-y) \Lambda(x,y) dy \, \partial_{\lambda_i} u_\lambda(x) - \theta(\lambda) \, \partial_{\lambda_i} u_\lambda(x)  \\ \displaystyle
 = - \int_{\mathbb{R}^d} a(x-y) \Lambda(x,y) (y_i - x_i) e^{\lambda(y-x)}u_\lambda(y) dy + \big(\partial_{\lambda_i} \theta(\lambda) \big) u_\lambda (x).
\end{array}
\end{equation}
The solvability condition for \eqref{specA-dif1-bis} reads
\begin{equation}\label{solv1}
\begin{array}{l}
\displaystyle
\int_{\mathbb{T}^d} \int_{\mathbb{R}^d} a(x-y) \Lambda(x,y) (y_i - x_i) e^{\lambda(y-x)} u_\lambda(y) u_\lambda^\star (x) dy dx
\\[2mm] = \displaystyle \partial_{\lambda_i} \theta(\lambda)   \int_{\mathbb{T}^d} u_\lambda (x)  u_\lambda^\star (x) dx = \partial_{\lambda_i} \theta(\lambda).
\end{array}
\end{equation}
Differentiating \eqref{specA-dif1} one more time in $\lambda_j$ yields
\begin{equation}\label{specA-dif2}
\begin{array}{l}
\displaystyle
\int_{\mathbb{R}^d} a(x-y) \Lambda(x,y) (y_i - x_i) (y_j - x_j) e^{\lambda(y-x)} u_\lambda(y) dy \\
\displaystyle
+ \int_{\mathbb{R}^d} a(x-y) \Lambda(x,y) (y_i - x_i) e^{\lambda(y-x)}  \partial_{\lambda_j} u_\lambda(y) dy
 \\[3.4mm]  \displaystyle
+ \int_{\mathbb{R}^d} a(x-y) \Lambda(x,y) (y_j - x_j) e^{\lambda(y-x)}  \partial_{\lambda_i} u_\lambda(y) dy
\\[3.4mm]
\displaystyle
+  \int_{\mathbb{R}^d} a(x-y) \Lambda(x,y) e^{\lambda(y-x)} \partial_{\lambda_i}  \partial_{\lambda_j} u_\lambda(y) dy \\[3.4mm]
\displaystyle
 - \int_{\mathbb{R}^d} a(x-y) \Lambda(x,y) dy \, \partial_{\lambda_i}  \partial_{\lambda_j} u_\lambda(x)
  = \partial_{\lambda_i}  \partial_{\lambda_j} \theta(\lambda) \, u_\lambda (x)
 \\[3mm]
 \displaystyle
  +  \partial_{\lambda_i} \theta(\lambda)  \, \partial_{\lambda_j}  u_\lambda (x) + \partial_{\lambda_j} \theta(\lambda)  \, \partial_{\lambda_i}  u_\lambda (x)
+   \theta(\lambda) \,\partial_{\lambda_i}  \partial_{\lambda_j}  u_\lambda(x).
\end{array}
\end{equation}
After rearranging  \eqref{specA-dif2} in the same way as \eqref{specA-dif1-bis}
the  solvability condition for \eqref{specA-dif2} reads
\begin{equation}\label{solv2}
\begin{array}{l}
\displaystyle
\int_{\mathbb{T}^d} \int_{\mathbb{R}^d} a(x-y) \Lambda(x,y) (y_i - x_i) (y_j - x_j) e^{\lambda(y-x)} u_\lambda(y)  u_\lambda^\star (x) dy dx \\[3.3mm]
\displaystyle
+  \int_{\mathbb{T}^d}  \int_{\mathbb{R}^d} a(x-y) \Lambda(x,y) (y_i - x_i) e^{\lambda(y-x)}  \partial_{\lambda_j} u_\lambda(y)  u_\lambda^\star (x)  dy dx \\[3.3mm]
\displaystyle
+  \int_{\mathbb{T}^d}  \int_{\mathbb{R}^d} a(x-y) \Lambda(x,y) (y_j - x_j) e^{\lambda(y-x)}  \partial_{\lambda_i} u_\lambda(y)  u_\lambda^\star (x)  dy dx
 \\ [3mm]
\displaystyle
 - \partial_{\lambda_i} \theta(\lambda)   \int_{\mathbb{T}^d} \partial_{\lambda_j} u_\lambda (x) \,  u_\lambda^\star (x) dx -  \partial_{\lambda_j} \theta(\lambda)   \int_{\mathbb{T}^d} \partial_{\lambda_i} u_\lambda (x) \,  u_\lambda^\star (x) dx
 \\ [3mm]
\displaystyle
 = \partial_{\lambda_i} \partial_{\lambda_j} \theta(\lambda).
\end{array}
\end{equation}
At $\lambda = 0$  relation \eqref{solv2}  takes the form
\begin{equation}\label{solv2at0}
\begin{array}{l}
\displaystyle
\partial_{\lambda_i} \partial_{\lambda_j} \theta(0)
= \int_{\mathbb{T}^d} \int_{\mathbb{R}^d} a(x-y) \Lambda(x,y) (y_i - x_i) (y_j - x_j) u_0^\star (x) dy dx \\
\displaystyle
+ \int_{\mathbb{T}^d}  \int_{\mathbb{R}^d} a(x-y) \Lambda(x,y) (y_i - x_i) \partial_{\lambda_j} u_0(y)  u_0^\star (x)  dy dx
 \\[3.3mm]
\displaystyle
+ \int_{\mathbb{T}^d}  \int_{\mathbb{R}^d} a(x-y) \Lambda(x,y) (y_j - x_j) \partial_{\lambda_i} u_0(y)  u_0^\star (x)  dy dx\\ [3mm]
\displaystyle
 - \partial_{\lambda_i} \theta(0)   \int_{\mathbb{T}^d} \partial_{\lambda_j} u_0 (x) \,  u_0^\star (x) dx - \partial_{\lambda_j} \theta(0)   \int_{\mathbb{T}^d} \partial_{\lambda_i} u_0 (x) \,  u_0^\star (x) dx.
\end{array}
\end{equation}

\begin{lemma}
The matrix $\nabla\nabla \theta(0)$ 
is positive definite.
\end{lemma}

\begin{proof}
Notice that the matrix defined on the right-hand side of \eqref{solv2at0}
coincides with the symmetric part of the effective diffusion matrix
\begin{equation}\label{Meff}
\begin{array}{l}
\displaystyle
\Theta^{ij} = \frac12 \int_{\mathbb{T}^d} \int_{\mathbb{R}^d} a(x-y) \Lambda(x,y) (y_i - x_i) (y_j - x_j) u_0^\star (x) dy dx \\
\displaystyle
- \int_{\mathbb{T}^d}  \int_{\mathbb{R}^d} a(x-y) \Lambda(x,y) (x_i - y_i) \varkappa_j (y)  u_0^\star (x)  dy dx
 + b_i   \int_{\mathbb{T}^d} \varkappa_j (x) \,  u_0^\star (x) dx,
\end{array}
\end{equation}
that was constructed in \cite{PZh2019} for 
the convolution type operator $A_0$.

Indeed, at $\lambda = 0$ relation \eqref{solv1} takes the form
\begin{equation}\label{solv1-at0}
\partial_{\lambda_i} \theta(0) = \int_{\mathbb{T}^d} \int_{\mathbb{R}^d} a(x-y) \Lambda(x,y) (y_i - x_i) u_0^\star (x) dy dx,
\end{equation}
where $u_0^\star$ is the eigenfunction of the adjoint operator $A_0^\star$ corresponding to the principal eigenvalue $\theta(0) = 0$.
Observe that the expression on the right-hand side of \eqref{solv1-at0} taken with the negative sign,  coincides with that for the $i$-th coordinate of the effective drift $b_i$ of the operator $A_0$, see \cite{PZh2019}. That is
\begin{equation}\label{b}
\partial_{\lambda_i} \theta(0) = - b_i.
\end{equation}
Letting $\lambda=0$ in \eqref{specA-dif1-bis}, substituting \eqref{b} into \eqref{specA-dif1-bis}, considering the relation $u_0(x) \equiv 1$ and recalling the equation for the corrector $\varkappa$, see \cite{PZh2019},  we conclude that
\begin{equation}\label{varkappa1}
\partial_{\lambda_i} u_\lambda (x) \big|_{\lambda = 0} =  \varkappa_i (x).
\end{equation}
Finally, by \eqref{b} and \eqref{varkappa1} we obtain $\partial_{\lambda_i} \partial_{\lambda_j} \theta(0)=
\Theta^{ij} +\Theta^{ji} $. Then positive definiteness of the matrix $\nabla \nabla \theta(0)$ follows from  \cite [Proposition 6.1]{PZh2019}.
\end{proof}
\medskip

We turn to the case  $\lambda=\lambda_0 + r$ with $\lambda_0 \neq 0, \ \lambda_0 \in \Gamma$, and $r$
belonging to a small neighbourhood of the origin. Then
\begin{equation}\label{A_lr}
A_\lambda u (x)\! =\!\! \int\limits_{\mathbb{R}^d}\!\!\! a(x\!-\!y) \Lambda(x,y) e^{\lambda_0 (y-x)} e^{r (y-x)} u(y) dy -\!\! \int\limits_{\mathbb{R}^d}\!\!\! a(x\!-\!y) \Lambda(x,y) dy \, u(x).
\end{equation}
Let us consider the operator
$
\tilde A_\lambda = R^{-1}_{\lambda_0} A_\lambda R_{\lambda_0},
$
where $R_{\lambda_0} f(x) = u_{\theta(\lambda_0)}(x) f(x)$  is the operator of multiplication by the principal eigenfunction $u_{\theta(\lambda_0)}$ of the operator $A_{\lambda_0}$.\\
The operators $A_\lambda$ and $\tilde A_\lambda$ are similar, thus they have the same spectrum. In particular, the spectral problem for $\tilde A_\lambda$ reads
\begin{equation}\label{specA-r}
\begin{array}{l}
\displaystyle
\int_{\mathbb{R}^d} a(x-y) \Lambda(x,y) u^{-1}_{\theta(\lambda_0)}(x) u_{\theta(\lambda_0)}(y) e^{\lambda_0 (y-x)} e^{r (y-x)} v(y) dy
\\ \displaystyle
- \int_{\mathbb{R}^d} a(x-y) \Lambda(x,y) dy \, v(x) = \theta(\lambda) v(x),
\end{array}
\end{equation}
where $\theta(\lambda)$ is the principal eigenvalue of $A_\lambda$.
Denote
\begin{equation}\label{15A}
\theta_{\lambda_0}(r)=\theta(\lambda)-\theta(\lambda_0) \quad\hbox{with }r= \lambda - \lambda_0.
\end{equation}
For $\lambda = \lambda_0$ we have from \eqref{A_lr}:
\begin{equation}\label{specA-r-at0}
\begin{array}{l}
\displaystyle
A_{\lambda_0} u_{\theta(\lambda_0)} (x) = \int_{\mathbb{R}^d} a(x-y) \Lambda(x,y) e^{\lambda_0 (y-x)} u_{\theta(\lambda_0)}(y) dy  \\ \displaystyle
- \int_{\mathbb{R}^d} a(x-y) \Lambda(x,y) dy \, u_{\theta(\lambda_0)}(x) = \theta(\lambda_0)  \, u_{\theta(\lambda_0)}(x).
\end{array}
\end{equation}
Dividing this equation by  $ u_{\theta(\lambda_0)}(x)$ we get
\begin{equation}\label{specA-r-bis}
\begin{array}{l}
\displaystyle
\int_{\mathbb{R}^d} a(x-y) \Lambda(x,y) u^{-1}_{\theta(\lambda_0)}(x) u_{\theta(\lambda_0)}(y) e^{\lambda_0 (y-x)} dy  \\ \displaystyle
= \int_{\mathbb{R}^d} a(x-y) \Lambda(x,y) dy + \theta(\lambda_0).
\end{array}
\end{equation}
Thus  \eqref{specA-r}, \eqref{15A} and \eqref{specA-r-bis} imply
\begin{equation}\label{specA-r-bisbis}
\begin{array}{l}
\displaystyle
\int_{\mathbb{R}^d} a(x-y) \Lambda(x,y) u^{-1}_{\theta(\lambda_0)}(x) u_{\theta(\lambda_0)}(y) e^{\lambda_0 (y-x)} e^{r (y-x)} v(y) dy  \\ \displaystyle
= \Big[ \int_{\mathbb{R}^d} a(x-y) \Lambda(x,y) dy + \theta(\lambda_0) \Big] v(x) +  \theta_{\lambda_0}(r) v(x)
 \\ \displaystyle
 =\int_{\mathbb{R}^d} a(x-y) \Lambda(x,y) u^{-1}_{\theta(\lambda_0)}(x) u_{\theta(\lambda_0)}(y) e^{\lambda_0 (y-x)} dy \, v(x)  +  \theta_{\lambda_0}(r) v(x).
\end{array}
\end{equation}
This spectral problem is similar to that in  \eqref{A_l}, if we replace the kernel $a(x-y)\Lambda(x,y)$
with the kernel
$$
a^{(\lambda_0)}(x-y) \Lambda^{(\lambda_0)}(x,y) = a(x-y) e^{\lambda_0 (y-x)} \, \Lambda(x,y) u^{-1}_{\theta(\lambda_0)}(x) u_{\theta(\lambda_0)}(y).
$$
According to \eqref{15A},
$$
\frac{\partial^2 \theta (\lambda_0) }{\partial \lambda_i \partial \lambda_j} = \frac{\partial^2 \theta_{\lambda_0}(0)}{\partial r_i \partial r_j},
$$
and the desired positive definiteness follows.
\end{proof}
\medskip

\begin{remark}
The structure of the set $\Gamma = \{ \lambda \in \mathbb{R}^d: \ \theta(\lambda) > -g_{\rm min} \}$ depends on the kernel $ a(x-y) \Lambda(x,y)$ of the operator $A_0$. For example, if $a(-z) = a(z)$ and $\Lambda(x,y)$ is a symmetric periodic function, then $\theta(-\lambda) = \theta(\lambda)$ and $\theta(0) = 0$ is the minimum of $\theta(\lambda)$ (as a function of $\lambda$). Consequently, in this case $\Gamma = \mathbb{R}^d$ and $\theta(\lambda)\ge 0$ for all $\lambda$.

Also, $\Gamma=\mathbb R^d$ if $\Lambda=\Lambda(x-y)$. In this case the spectrum of $A_\lambda$ is
discrete for any $\lambda\in\mathbb R^d$.
\end{remark}

The following example illustrates that in general the set $\Gamma$ need not coincide with $\mathbb{R}^d$.

\medskip
\noindent
{\bf Example.} 
Take $a(z) = {\bf 1}_{[-\frac12,\frac12]^d}$ equal to the characteristic function of the period, and $\Lambda(x,y) = b(x) \Lambda_0 (x-y)$. We assume that $\Lambda_0(z)$ is a smooth periodic function, $0<\alpha_1 \le \Lambda_0(z) \le \alpha_2 < \infty$, and $\Lambda_0$ has the form of a single peak:
$$
\Lambda_0 (z) = \left\{
\begin{array}{l}
\alpha_2, \quad |z-z_0| < \frac{c}{2} \\
\alpha_1, \quad |z - z_0|> c
\end{array}
\right.
$$
Here $z_0 \neq 0, z \in \mathbb{T}^d$, and we choose sufficiently small constants $\alpha_1$ and $c$ and sufficiently large constant $\alpha_2$ so that the following normalization condition holds:
$$
\int_{\mathbb{R}^d} a(z) \Lambda_0(z) dz = \int_{\mathbb{T}^d} a(z) \Lambda_0(z) dz = 1.
$$

Then the spectral problem \eqref{specA} for $A_\lambda$ reads
$$
\begin{array}{l}
\displaystyle
b(x) \,\int_{\mathbb{T}^d} a(x-y) \Lambda_0 (x- y) e^{\lambda(y-x)} u_\lambda(y) dy  \\ \displaystyle
= b(x) \, \int_{\mathbb{T}^d} a(x-y) \Lambda_0(x-y) dy \, u_\lambda(x)
+ \theta(\lambda) u_\lambda (x),
\end{array}
$$
and, after straightforward rearrangements,
\begin{equation}\label{ex-1}
\frac{b(x)}{b(x)+ \theta(\lambda)} \int_{\mathbb{T}^d} a(x-y) \Lambda_0 (x- y) e^{\lambda(y-x)} u_\lambda(y) dy = u_\lambda (x),
\end{equation}
where $u_\lambda > 0$ is the principal eigenfunction.

We now take a periodic positive function $b(x), \; 0< b_{\rm min} \le b(x) \le 1$, such that
\begin{equation}\label{ex-2}
\big\| \frac{b(x)}{b(x) -  b_{\rm min}} \big\|_{L^2(\mathbb{T}^d)} < 1+\delta \quad \mbox{with } \; 0<\delta<1.
\end{equation}
Obviously, inequality \eqref{ex-2}  remains valid for  $ \frac{b(x)}{b(x) + \theta(\lambda)}$ with any $\theta(\lambda)>  -  b_{\rm min}$.
Then the operator on the left hand side of equation \eqref{ex-1} is positive and compact in $L^2(\mathbb{T}^d)$.

Assuming that $\alpha_1$ is small enough we conclude that there exists $\lambda_0$ such that $\lambda_0 z_0>0$ and
\begin{equation}\label{ex-3}
0< a(z) \Lambda_0(z) e^{-\lambda_0 z}<{\textstyle \frac12 }\quad \mbox{for all } \; z \in \mathbb{T}^d.
\end{equation}
Then from \eqref{ex-2}, \eqref{ex-3} it follows that the $L^2(\mathbb T^d)$ norm of the left-hand side in \eqref{ex-1} is strictly less than $\| u_{\lambda_0}\|_{L^2(\mathbb{T}^d)}$.  Therefore, equation  \eqref{ex-1} has no positive solution $u(x) \in L^2(\mathbb{T}^d)$, and there are no points of the discrete spectrum of $A_{\lambda_0}$ located above the continuous spectrum, that is
$$
\sigma_{\rm disc}(A_{\lambda_0})\cap ( -g_{\rm min},+\infty)=\emptyset.
$$
Observe that in this example equation \eqref{ex-1} has no positive solutions for all $\lambda$ situated
in a sufficiently small neighbourhood of $\lambda_0$, thus $\lambda_0$ is an interior point of
$\mathbb R^d\setminus\Gamma$.

\subsection{Properties of the Hamiltonian}\label{ss_ha_and la}
Denote
\begin{equation}\label{conseq}
H(\lambda) := \mathtt{s}(A_\lambda) = \left\{
\begin{array}{l}
\theta(\lambda), \; \lambda \in \Gamma \\
- g_{\min}, \ \mbox{otherwise}
\end{array}
\right.
\end{equation}
As a consequence of  Theorem \ref{SP} we have
\begin{proposition}
The function $H(\cdot)$ is convex. It is strictly convex on the set $\Gamma$.
Moreover,
\begin{equation}\label{htoinfty}
\frac{H(\lambda)}{|\lambda|} \to +\infty \quad \mbox{as } \; |\lambda| \to + \infty.
\end{equation}
\end{proposition}
\begin{proof}
The convexity and the strict convexity on $\Gamma$ have been proved in Theorem \ref{SP}.
The  relation in  \eqref{htoinfty} follows from Lemma \ref{l_exp_growth}.
\end{proof}
By Lemma \ref{l_growth_a_lambda} we have
\begin{equation}\label{H}
\lim_{\varepsilon \to 0} \varepsilon \, \ln \mathbb{E} e^{\frac{\lambda}{\varepsilon} \xi_0^\varepsilon(t)}
= t \, H(\lambda),
\end{equation}
with $H(\lambda)$ defined in \eqref{conseq}. 

\smallskip\noindent
Concluding this subsection we summarize the properties of the function 
$H(\lambda)$:\\
1) $H(\lambda)$ is convex, it is strictly convex for $\lambda \in \Gamma$, \\
2) $H(0) = 0$ and  $H(\lambda)$ is strictly convex at $\lambda = 0$, \\
3) $\frac{H(\lambda)}{|\lambda|} \to +\infty$ as $|\lambda| \to + \infty$, \\
4) the function $H(\lambda)$ equals to a constant on the set $\lambda \in \Upsilon = \mathbb{R}^d \setminus \Gamma$:
$$
H(\lambda) = -   g_{\min}, \quad \lambda \in \Upsilon = \mathbb{R}^d \setminus \Gamma,
$$
the set $\Upsilon \subset \mathbb{R}^d$ is bounded and convex.   If the interior of $\Upsilon$ is not empty,
then the boundary $\partial \Upsilon$ is Lipschitz continuous.

\subsection{The Legendre transform of $H(\lambda)$ and the G\"{a}rtner-Ellis theorem.}

Let $L$ and $L_t$ be the Legendre transform of $H(\cdot)$ and $H_t:=tH$, respectively,  i.e.
\begin{equation}\label{def_ham_lag}
L(\zeta) = \sup_\lambda \big( \lambda \zeta - H(\lambda) \big),
\qquad
L_t(\zeta) = \sup_\lambda \big( \lambda \zeta - tH(\lambda) \big)=t L\big(\frac \zeta t\big),\quad \zeta \in \mathbb{R}^d.
\end{equation}
We recall (see, for instance, \cite{DeZe}) that  $\zeta'\in \mathbb R^d$ is an exposed point of $L$ if for some $\theta\in\mathbb R^d$ and
all $\zeta\not=\zeta'$,
$$
\theta\cdot \zeta-L(\zeta)>\theta\cdot \zeta'-L(\zeta').
$$
The properties of $H(\lambda)$ imply the following properties of $L(\zeta)$: \\
1) $L(\zeta)$ is a convex function, $L(\zeta)<+\infty$ for any $\zeta\in\mathbb R^d$. It is strictly convex
in the neighbourhood of infinity, that is there exists $R_0$ such that $L(\zeta)$ is strictly convex
for all $\zeta$ such that $|\zeta|\geq R_0$, \\
2) $L(\zeta)$ is non-negative: $L(\zeta) \ge 0$,\\
3) $\min L(\zeta) = L(\zeta^\ast) = 0$ and $L$ is strictly convex at $\zeta^\ast$, \\
4) $\frac{L(\zeta)}{|\zeta|} \to +\infty$ as $|\zeta| \to + \infty$, in particular, $L(\zeta)$ has compact sub-level sets, \\
5) 
The complement to the set of exposed points of $L$, if not empty, consists of segments of bounded length
with one end at  $0$, the restriction of $L$ on each such segment is a linear function.   

\medskip\noindent
Denote the set of exposed points of $L$ by $\Omega$.
It should be emphasized that the origin need not be an exposed point of $L(\cdot)$. In particular,  the restriction 
of $L$ on two segments going from the origin in the opposite directions can form the same linear function.
However, if $\mathbb R^d\setminus\Gamma$ has a non-trivial interior, then $0\in \Omega$.
 This can be justified by the  convex analysis arguments if we take into account the properties
of $H(\cdot)$.


\begin{theorem}\label{l_findim_per}
For any $t>0$ and any $x^0\in\mathbb R^d$ the random vector $\xi_{x^0}^\eps(t)-x_0$ satisfies the large deviation principle
with the rate function   $L_t(x)=tL\big(\frac{x}t\big)$.
\end{theorem}
\begin{proof}
As an immediate  consequence of formula  \eqref{H} we obtain
\begin{equation}\label{H_bibis}
\lim_{\varepsilon \to 0} \varepsilon \, \ln \mathbb{E}
e^{\frac{\lambda}{\varepsilon} (\xi_{x^0}^\varepsilon(t)-x^0)} = t \, H(\lambda),
\end{equation}
Then the upper large deviation bound follows from the G\"artner-Ellis theorem.  We have
$$
\lim\limits_{\delta\to0}\lim\limits_{\eps\to0}\eps\log\big[ \mathbb P\big\{|(\xi_{x^0}^\varepsilon(t)-x^0)-x|\leq\delta\big\}\big]
\leq -L_t (x).
$$

The lower bound is slightly more tricky. By the  G\"artner-Ellis theorem for any $t>0$ and any $x\in\mathbb R^d$
such that $\frac{x} t$ is an exposed point of $L(\cdot)$ the inequality
$$
\lim\limits_{\delta\to0}\lim\limits_{\eps\to0}\eps\log\big[ \mathbb P\big\{|(\xi_{x^0}^\varepsilon(t)-x^0)-x|\leq\delta\big\}\big]
\geq -L_t (x)
$$
holds.  Without loss of generality we assume that $x^0=0$.
We first assume that 
$0\in\Omega$.
Consider $x\in\mathbb R^d$ which is a non exposed point of $L_t(\cdot)$ and
represent it as $x=r\phi$ with $\phi\in S^{d-1}$ and $r>0$.
{
Since $\xi(\cdot)$ is a Markov process, for any $\kappa\in (0,1)$ and for any $\delta>0$ we have
$$
\mathbb P\big\{|\xi_0^\varepsilon(t)-x|\leq2\delta\big\}=
\mathbb P\big\{|\xi_0^\varepsilon(t)-r\phi|\leq2\delta\big\}
$$
\begin{equation}\label{mark_expl_1}
\geq
{\textstyle \mathbb P\big\{\{|\xi_0^\varepsilon(\kappa t)|\leq\delta\}\cap
\{|\xi_0^\varepsilon(t)-\xi_0^\varepsilon(\kappa t)-r\phi|\leq \delta\}\big\}
}
\end{equation}
$$
\geq  \mathbb P\big\{\{|\xi_0^\varepsilon(\kappa t)|\leq\delta\}
\min\limits_{|y|\leq\delta}\mathbb P\big\{|\xi_y^\varepsilon((1-\kappa)t)-y-r\phi|\leq2\delta\big\}
$$
Denote by $R$ the length of the segment
$(0,R\phi)=(\mathbb R^d\setminus\Omega)\cap \{(0,s\phi)\,:\,s>0\}$.  Then, for any $h_0>0$, the point
$(R+h_0)\phi$ is exposed for $L_t$.
Therefore,  choosing  $\kappa$ in \eqref{mark_expl_1} so that
$\frac r{1-\kappa}=R+h_0$, that is
 $\kappa=\frac{R+h_0- r}{R+h_0}$, and applying
the G\"artner-Ellis theorem, we arrive for all sufficiently small  $\delta>0$ and $h_0>0$ at the following lower bound:
$$
\mathbb P\big\{|\xi_0^\varepsilon(t)-x|\leq2\delta\big\}
\geq \exp\big[-{\textstyle \big(\frac{R+h_0- r}{R+h_0} L_t(0)-\psi(\delta)\big)(1+o(1))}\big]\ \times
$$
$$
\exp\big[ -{\textstyle \big(\frac{ r}{R+h_0} L_t((R+h_0)\phi)-\psi(\delta)\big)(1+o(1))}\big]
$$
$$
\geq \exp\big[-{\textstyle \big(\frac{R- r}{R} L_t(0)+
\frac{ r}{R} L_t(R\phi)-C_Lh_0-2\psi(\delta)\big)(1+o(1))\big]
}
$$
$$
= \exp\big[-{\textstyle \big( L_t(r\phi)-C_Lh_0-2\psi(\delta)\big)(1+o(1))\big]},
$$
where $o(1)$ tends to zero as $\eps\to0$, $\psi(\delta)\to0$ as $\delta\to 0$,
and $C_L$ is a constant which only depends on $L(\cdot)$; we have used here  the fact that $L_t(\cdot)$ is linear on the segment $[0,R\phi]$.
This implies the desired lower bound.
}

 If $0$ is not an exposed point then
 there is a segment that passes through
$0$, such that $L_t$ is linear on this segment, and there are exposed points of $L_t$ in the intersections of any
neighbourhoods of the end points of this segment with the straight line that contains the segment.
In this case in the same way as above one can show that
 $$
\lim\limits_{\delta\to0} \lim\limits_{\eps\to0}\eps\log\big[{\textstyle \mathbb P\big\{\{|\xi_0^\varepsilon(t)|\leq\delta\}}\big]\geq
-L_t(0) .
 $$
It remains to use one more time the same arguments as in the previous case to obtain the required lower bound
for any $x\in\mathbb R^d$.
 This completes the proof of Theorem.
\end{proof}

\subsection{Large deviation principle in the paths space}\label{ss_ldp_pureper}

The goal of this section is to show that the process $\xi_x^\eps(\cdot)$  satisfies on any time interval $[0,T]$
the large deviation principle in the paths space $\mathbf{D}([0,T];\mathbb R^d)$ with the rate function
defined by
\begin{equation}\label{rate_f_per}
I(\gamma(\cdot))=\left\{\begin{array}{ll}\displaystyle
\int_0^T L(\dot\gamma(t))\,dt,& \hbox{\rm if }\gamma\ \hbox{\rm is absolutely continuous and }\gamma(0)=x,\\[3mm]
+\infty,& \hbox{\rm otherwise},
\end{array}
\right.
\end{equation}
where $L(\cdot)$ is introduced in \eqref{def_ham_lag}. 
An important property of $I(\cdot)$ is the compactness of its sublevel sets in the topology of uniform convergence
in $C([0,T];\mathbb R^d)$.
\begin{lemma}\label{l_compa_in_c}
The set $\{\gamma\in C([0,T];\mathbb R^d)\,:\,I(\gamma)\leq s, \gamma(0)=x\}$ is compact in $C([0,T];\mathbb R^d)$
for any $s\in\mathbb R$ and any $x\in\mathbb R^d$.
\end{lemma}
\begin{proof}
  This statement is an immediate consequence of the Arzel\`a-Ascoli theorem and the relation
$\lim\limits_{|\zeta|\to\infty} \frac{L(\zeta)}{|\zeta|}=\infty$.
\end{proof}

The next statement is also important for the further analysis.
\begin{proposition}\label{p_no_zigzag}
Let $\xi^\eps_x$ be a Markov process with the generator $A^\eps$ that satisfies conditions \eqref{M1}--\eqref{lm}, and assume that $\gamma(\cdot)$ is an absolutely continuous function,
$\gamma(0)=x$.
Then for any $M>0$ there exists a function $\delta_0(\delta)$, $\delta_0:(0,1]\mapsto \mathbb R^+$
such that $\delta_0(\delta)\to0$ as $\delta\to0$, and for any $\pi\in\mathcal{K}$ with $\ell(\pi)\leq\delta$
we have
$$
\mathbb P\big\{\{\sup\limits_{0\leq t\leq T}|\xi_x^\eps(t)-\gamma(\pi(t))|\geq \delta_0\}\cap\{
|\xi_x^\eps(j\delta)-\gamma(\pi(j\delta))|\leq\delta,\,{\textstyle  j=0,\ldots,\frac T\delta} \}\big\}\leq \exp\big\{{\textstyle -\frac M\eps}\big\}
$$
 for all sufficiently small $\eps>0$.  Moreover, for any $s>0$ and for all sufficiently small $\eps>0$,
 $$
\sup\limits_{\gamma\in\Phi(s)}\mathbb P\big\{\{\sup\limits_{0\leq t\leq T}|\xi_x^\eps(t)-\gamma(\pi(t))|\geq \delta_0\}\cap\{
|\xi_x^\eps(j\delta)-\gamma(\pi(j\delta))|\leq\delta,\,{\textstyle  j=0,\ldots,\frac T\delta} \}\big\}\leq \exp\big\{{\textstyle -\frac M\eps}\big\},
$$
where $\Phi(s)=\{\gamma\in \mathbf{D}([0,T],\mathbb R^d)\,:\,I(\gamma)\leq s, \,\gamma(0)=x\}$.
\end{proposition}
\begin{proof}
Consider an auxiliary process $\eta^{\eps}(\cdot)$ with generator
$$
A_{\rm sym}^\eps v(x) =\frac1{\eps^{d+1}}\int_{\mathbb R^d} a_s\big({\textstyle\frac{x-y}\eps}\big)(v(y)-v(x))\,dy,
$$
where
$$
a_s(z)=\mathtt{C}_0 e^{- k |z|^p}, \quad  \mathtt{C}_0=\Lambda^+\mathtt{C},
$$
with the same $p$, $k$, $\mathtt{C}$ and $\Lambda^+$ as those in \eqref{lt} and \eqref{lm}.
For the transition densities of the processes $\xi_x^\eps(\cdot)$ and $\eta_x^{\eps}(\cdot)$ we use the notation
$q^\eps(x,y,t)$ and $q^\eps_s(x,y,t)$, respectively. We also define a function $q^\eps_+(x,y,t)$ as the solution of the following problem
$$
\partial_t q(x,y,t)=\frac1{\eps^{d+1}}\int_{\mathbb R^d}a_s\big({\textstyle\frac{y-z}\eps}\big)q(x,z,t)\,dz,\qquad
q(0,x,y)=\delta(y-x).
$$
By the maximum principle we have
\begin{equation}\label{max_prii}
q^\eps(x,y,t)\leq q^\eps_+(x,y,t) \qquad\hbox {for all }x,\,y\in\mathbb R^d\ \hbox{and all }t\geq 0.
\end{equation}
It is also clear that
$$
q_+^{\eps}(x,y,t) =\exp\big({\textstyle \frac{\mathtt{C}_1t}\eps}\big)q_s^{\eps}(x,y,t) \quad\hbox{with }
\mathtt{C}_1=\int_{\mathbb R^d}\mathtt{C}_0\exp(-k|z|^p)\,dz.
$$
The Hamiltonian and the Lagrangian that correspond to the process $\eta^{\eps}$ are defined in the same way
as in the previous section.  Namely,
$$
H^s(\lambda)=\int_{\mathbb R^d}\mathtt{C}_0\exp(-\lambda\cdot z-k|z|^p)\,dz-\mathtt{C}_1,\quad
L^s(\zeta)=\max\limits_{\lambda\in\mathbb R^d}(\zeta\cdot\lambda-H^s(\lambda)).
$$
One can easily check that both $H^s$ and $L^s$ are smooth strictly convex functions and, moreover,
$\frac{L^s(\zeta)}{|\zeta|}\to+\infty$ as $|\zeta|\to\infty$.

Considering the continuity of $\gamma(\cdot)$ we can construct a function $\delta_0(\delta)$ such that
\begin{itemize}
  \item $\delta_0(\delta)\to0$ as $\delta\to0$.
  \item $|\gamma(t')-\gamma(t'')|\leq \frac14\delta_0$ if $|t'-t''|\leq3\delta$.
  \item $\min\limits_{\phi\in S^{d-1}}\big\{\delta L^s\big(\frac{\delta_0\phi}{2\delta}\big)\big\}\to+\infty$
  as $\delta\to0$.
\end{itemize}
\begin{lemma}\label{l_est_sproc}
For any $\delta_0$ and any $\tau>0$ we have
$$
\mathbb{P}\{\sup\limits_{0\leq t\leq\tau}|\eta_x^{\eps}(t)-x|\geq \delta_0\}\leq
2 \mathbb{P}\{|\eta_x^{\eps}(\tau)-x|\geq \delta_0\}.
$$
\end{lemma}
\begin{proof}
Denote by $\mathcal{E}_0$ and $\mathcal{E}_1$ the events
$$
\mathcal{E}_0=\{\sup\limits_{0\leq t\leq\tau}|\eta_x^{\eps}(t)-x|\geq \delta_0\},\quad
\mathcal{E}_1=\{|\eta_x^{\eps}(\tau)-x|\leq \delta_0\}
$$
Both $\mathcal{E}_0$ and $\mathcal{E}_1$ depend on $\eps$, however, we do not indicate this dependence
explicitly.
Due to the symmetry of $a_s(\cdot)$ by the Markov property we have
$$
\mathbb{P}(\mathcal{E}_0\cap\mathcal{E}_1)
=\mathbb{P}(\mathcal{E}_1\big|\mathcal{E}_0)\mathbb{P}(\mathcal{E}_0)
<\frac12\mathbb{P}(\mathcal{E}_0).
$$
Therefore,
$$
\mathbb{P}(\mathcal{E}^c_1)=\mathbb{P}(\mathcal{E}_0\cap\mathcal{E}^c_1)>\frac12\mathbb{P}(\mathcal{E}_0),
$$
and the desired statement follows.
\end{proof}
By the G\"artner-Ellis theorem
 for all sufficiently small $\eps>0$ we have
$$
\mathbb{P}\{|\eta_x^{\eps}(\delta)-x|\geq \delta_0\}\leq \exp\big({\textstyle-\frac\delta\eps \min\limits_{\phi\in S^{d-1}}
L^s(\frac{\delta_0\phi}{\delta})}\big).
$$
For arbitrary $M>0$ we choose small enough $\delta>0$ such that $\min\limits_{\phi\in S^{d-1}} \delta
L^s(\frac{\delta_0(\delta)\phi}{2\delta})\geq 2M$. Then, for sufficiently small $\eps>0$ and for any
$\pi\in\mathcal{K}$ with $\ell(\pi)\leq\delta$,
$$
\mathbb P\big\{\{\sup\limits_{0\leq t\leq T}|\eta_x^{\eps}(t)-\gamma(\pi(t))|\geq \delta_0(\delta)\}\cap\{
|\eta_x^{\eps}(j\delta)-\gamma(\pi(j\delta))|\leq\delta,\,{\textstyle  j=0,\ldots,\frac T\delta} \}\big\}
$$
\begin{equation}\label{ineq_for_symm}
\leq\mathbb P\big\{{\textstyle \sup\limits_{0\leq t\leq \delta}|\eta_x^{\eps}(t+j\delta)-\eta_x^{\eps}(j\delta)|\geq
\frac{\delta_0(\delta)}2 \ \  \hbox{for some } j\leq\frac T\delta} \big\}
\end{equation}
$$
\leq {\textstyle\frac{T}\delta\exp\big(-\frac\delta\eps \min\limits_{\phi\in S^{d-1}}
L^s(\frac{\delta_0\phi}{2\delta})}\big)
\leq {\textstyle\frac{T}\delta\exp\big\{{\textstyle -\frac{2M}\eps}\big\}}\leq
\exp\big\{{\textstyle -\frac M\eps}\big\}.
$$

Next, for any partition of the interval $[0,T]$, $0\leq t_1\leq\ldots\leq t_{N_1}\leq T$,
and for any collection of domains ${\mathcal{B}_1},\ldots,{\mathcal{B}_{N_1}}$  the following inequality
holds:
$$
\mathbb P\big\{\bigcap\limits_{j=1}^{N_1}\{\xi_x^\eps(t_j)\in \mathcal{B}_j\big\}
$$
$$
=\int\limits_{\mathcal{B}_1}\!\!q^\eps(x,y^1,t_1)dy^1
\int\limits_{\mathcal{B}_2}\!\!q^\eps(y^1,y^2,t_2-t_1)dy^2\ldots
\int\limits_{\mathcal{B}_{N_1}}\!\!q^\eps(y^{N_1-1},y^{N_1},t_{N_1}-t_{N_1-1})dy^{N_1}
$$
$$
\leq \int\limits_{\mathcal{B}_1}\!\!q_+^{\eps}(x,y^1,t_1)dy^1
\int\limits_{\mathcal{B}_2}\!\!q_+^{\eps}(y^1,y^2,t_2-t_1)dy^2\ldots
\int\limits_{\mathcal{B}_{N_1}}\!\!q_+^{\eps}(y^{N_1-1},y^{N_1},t_{N_1}-t_{N_1-1})dy^{N_1}
$$
$$
\leq\exp\big({\textstyle\frac{\mathtt{C}_1\,T}\eps }\big)\times
$$
$$
\times
\int\limits_{\mathcal{B}_1}\!\!q_s^{\eps}(x,y^1,t_1)dy^1
\int\limits_{\mathcal{B}_2}\!\!q_s^{\eps}(y^1,y^2,t_2-t_1)dy^2\ldots
\int\limits_{\mathcal{B}_{N_1}}\!\!q_s^{\eps}(y^{N_1-1},y^{N_1},t_{N_1}-t_{N_1-1})dy^{N_1}
$$
$$
=\exp\big({\textstyle\frac{\mathtt{C}_1\,T}\eps }\big)\mathbb P\big\{\bigcap\limits_{j=1}^{N_1}\{\eta_x^{\eps}(t_j)\in \mathcal{B}_j\big\}
$$
Combining this inequality with  \eqref{ineq_for_symm} yields the first inequality  stated in Proposition.

In order to prove the second one it suffices to observe that, due to the compactness of the set $\Phi(s)$
in $C([0,T];\mathbb R^d)$, the function $\delta_0(\delta)$ can be chosen in such a way that
$|\gamma(t')-\gamma(t'')|\leq \frac14\delta_0$ if $|t'-t''|\leq3\delta$ for all $\gamma\in\Phi(s)$.
\end{proof}
\begin{proposition}\label{p_nonabsco}
For any $\gamma\in \mathbf{D}([0,T];\mathbb R^d)$, $\gamma(0)=x$,  that is not absolutely continuous
we have
$$
\lim\limits_{\delta\to0}\limsup\limits_{\eps\to0}\eps\log\big(\mathbb P\big\{\mathrm{dist}(\xi^\eps_x(\cdot),\gamma(\cdot))\leq\delta\big\}\big)=-\infty.
$$
\end{proposition}

\begin{proof}
 Consider  auxiliary operators defined by
$$
A^{\rm u} v(x)=\int_{\mathbb R^d}\Lambda^+a(x-y)v(y)\,dy-\Lambda^-v(x)\int_{\mathbb R^d}a(x-y)\,dy
$$
and
$$
A^+ v(x)=\int_{\mathbb R^d}\Lambda^+a(x-y)v(y)\,dy-\Lambda^+v(x)\int_{\mathbb R^d}a(x-y)\,dy
$$
and the corresponding scaled operators
$$
A^{{\rm u},\eps}v(x)=\frac1{\eps^{d+1}}\int_{\mathbb R^d}\Lambda^+a\big(\textstyle{\frac{x-y}\eps}\big)v(y)\,dy-\displaystyle{\frac1{\eps^{d+1}}
\Lambda^-v(x)\int_{\mathbb R^d}}a\big(\textstyle{\frac{x-y}\eps}\big)\,dy.
$$
and
$$
A^{+,\eps}v(x)=\frac1{\eps^{d+1}}\int_{\mathbb R^d}\Lambda^+a\big(\textstyle{\frac{x-y}\eps}\big)v(y)\,dy-\displaystyle{\frac1{\eps^{d+1}}
\Lambda^+v(x)\int_{\mathbb R^d}}a\big(\textstyle{\frac{x-y}\eps}\big)\,dy.
$$
Denote by $q^{{\rm u},\eps}(x,y,t)$, $q^{+,\eps}(x,y,t)$ and $q^{\eps}(x,y,t)$ the  solutions of the equations
$$
\partial_t q=A^{{\rm u},\eps}q,\quad
\partial_t q=A^{+,\eps}q\ \ \hbox{and }\ \partial_t q=A^{\eps}q,
$$
respectively, with the common initial condition $q(x,y,0)=\delta(y-x)$.\\
Since $\Lambda^+\geq \Lambda(x,y)$ and $\Lambda^-\leq \Lambda(x,y)$ for all $x$ and $y$ from $\mathbb R^d$,
 by the maximum principle we have
\begin{equation}\label{compar_max}
q^\eps(x,y,t)\leq q^{{\rm u},\eps}(x,y,t) \quad \hbox{for all }x,\,y\in\mathbb R^d\ \ \hbox{and }t>0.
\end{equation}
It is also clear that
$$
q^{{\rm u},\eps}(x,y,t) =\exp\big({\textstyle \frac{(\Lambda^+-\Lambda^-)\;t}\eps}\big)q^{+,\eps}(x,y,t)
$$
For an arbitrary partition $0\leq t_1<t_2<\ldots<t_N\leq T$ of the interval $[0,T]$, an arbitrary
set $x^1,\ldots,x^N$, $x^j\in\mathbb R^d$ and any $\delta>0$ we have
$$
\mathbb P\big\{\bigcap\limits_{j=1}^N\{|\xi^\eps(t_j)-x_j|\leq\delta\}\big\}
$$
$$
=\int\limits_{Q_\delta(x^1)}\!\!\!\!q^\eps(0,y^1,t_1)dy^1
\int\limits_{Q_\delta(x^2)}\!\!\!\!q^\eps(y^1,y^2,t_2-t_1)dy^2\ldots
\int\limits_{Q_\delta(x^N)}\!\!\!\!q^\eps(y^{N-1},y^N,t_N-t_{N-1})dy^N
$$
$$
\leq \int\limits_{Q_\delta(x^1)}\!\!\!\!q^{{\rm u},\eps}(0,y^1,t_1)dy^1
\int\limits_{Q_\delta(x^2)}\!\!\!\!q^{{\rm u},\eps}(y^1,y^2,t_2-t_1)dy^2\ldots
\int\limits_{Q_\delta(x^N)}\!\!\!\!q^{{\rm u},\eps}(y^{N-1},y^N,t_N-t_{N-1})dy^N
$$
$$
=\exp\big({\textstyle\frac{(\Lambda^+-\Lambda^-)\,T}\eps }\big)\times
$$
$$
\times
\int\limits_{Q_\delta(x^1)}\!\!\!\!q^{+,\eps}(0,y^1,t_1)dy^1
\int\limits_{Q_\delta(x^2)}\!\!\!\!q^{+,\eps}(y^1,y^2,t_2-t_1)dy^2\ldots
\int\limits_{Q_\delta(x^N)}\!\!\!\!q^{+,\eps}(y^{N-1},y^N,t_N-t_{N-1})dy^N.
$$
Let $\gamma$ be an arbitrary curve in ${\bf D}([0,T];\mathbb R^d)$ which is not absolutely continuous. Setting $x^j=\gamma(t_j)$,
taking uniform partitions of the interval $[0,T]$ and sending $N$ to infinity, from the last relation we deduce
$$
\mathbb P\big\{\sup\limits_{0\leq t\leq T}|\xi^\eps(t)-\gamma(\pi(t))|\leq\delta\}\big\}\leq
\exp\big({\textstyle\frac{(\Lambda^+-\Lambda^-)\,T}\eps }\big)
\mathbb P\big\{\sup\limits_{0\leq t\leq T}|\xi^{+,\eps}(t)-\gamma(\pi(t))|\leq\delta\}\big\};
$$
here $\xi^{+,\eps}(t)$ is a process with independent increments whose generator is $A^{+,\eps}$.

Due to \cite{Pukh94}, for any $\gamma$ that is not absolutely continuous this yields
\begin{equation}\label{non_abscont}
\lim\limits_{\delta\to0}\limsup\limits_{\eps\to0}\,\mathbb P\big\{\mathrm{dist}(\xi^\eps(\cdot),\gamma(\cdot))\leq\delta\}\big\}=-\infty=-I_\Lambda(\gamma).
\end{equation}
This implies the desired statement.
\end{proof}


The main result of this section reads.
\begin{theorem}
Let $\Lambda(x,y, \xi,\eta)=\Lambda(\xi,\eta)$, and assume that $\Lambda(\xi,\eta)$ is a measurable function for which conditions \eqref{M1}--\eqref{cond_perio} and \eqref{lm} are fulfilled. Then the process $\xi_x^\eps(t)$, $0\leq t\leq T$, satisfies in $\mathbf{D}([0,T]\,;\,\mathbb R^d) $ the large deviation principle with the rate function $I(\cdot)$ introduced in \eqref{rate_f_per}.\\ In particular,
for any $\gamma\in\mathbf{D}([0,T]\,;\,\mathbb R^d)$, $\gamma(0)=x$, the following relation holds:
\begin{equation}\label{dop_formulirov}
\lim\limits_{\delta\to0}\lim\limits_{\eps\to0}\eps\log \mathbb P\big\{ \mathrm{dist}(\xi^\eps_x(\cdot),
\gamma(\cdot))\leq \delta\big\}=-I(\gamma).
\end{equation}

\end{theorem}

\begin{proof}
 For any $\gamma(\cdot)$ that is not absolutely continuous  the relation
 $$
\lim\limits_{\delta\to0} \limsup\limits_{\eps\to0}\eps\log\big[{\textstyle \mathbb P\big\{\{\mathrm{dist}(\xi_x^\varepsilon(\cdot),\gamma(\cdot))\leq\delta\}}\big]=-\infty
 $$
 follows from Proposition \ref{p_nonabsco}.

 Assume that $\gamma(\cdot)$ is absolutely continuous, and $\int_0^T L(\dot\gamma)\,dt<+\infty$.
 We consider a piece-wise linear approximation of $\gamma$ defined by
 $$
  \gamma_N(t)=\left\{
  \begin{array}{ll}
  \gamma(t) &\hbox{if } t=0,\frac1N,\frac2N, \ldots, T\\
  \gamma(t_j)+(\gamma(t_{j+1})-\gamma(t_j))\frac{t-t_j}{t_{j+1}-t_j}&\hbox{if }t\in(t_j,t_{j+1}).
\end{array}
\right.
  $$
 For any $\varkappa>0$ there exists $N_0=N_0(\varkappa)$ such that for any $N\geq N_0$
 $$
 0\leq \int_0^T L(\dot\gamma(t))\,dt - \int_0^T L(\dot\gamma_N(t))\,dt\leq \varkappa.
 $$
Denote $\delta=\frac1N$.
 Then, by Proposition \ref{p_no_zigzag} there exists a function $\delta_0(\delta)$, $\delta_0:(0,1]\mapsto\mathbb R^+$,
 such that $\delta_0(\delta)\to0$ as $\delta\to0$,
 and for any $\pi\in\mathcal{K}$ with $\ell(\pi)\leq\delta$
 \begin{equation}\label{per_zazor}
 \begin{array}{c}
 \displaystyle
 \mathbb P\big\{\! \{|\xi_x^\eps(t_j)\!-\gamma(\pi(t_j))| \leq\delta,\,j=0,\ldots,N\}\cap
 \{\sup\limits_{0\leq t\leq T}|\xi_x^\eps(t)-\gamma(\pi(t))| \geq\delta_0\}\!\big\} \\[4mm]
 \leq\,
 \exp\big(-\frac M\eps(1+o(1))\big),
 \end{array}
 \end{equation}
 where $M=I(\gamma)+1$ and $o(1)\to0$ as $\eps\to0$.

In order to achieve the upper bound we fix $N\geq N_0$ and choose $\delta_1>0$ in such a way that
for any $\pi\in\mathcal{K}$ with $\ell(\pi)\leq\delta_1$
\begin{equation}\label{upp_perio1}
\begin{array}{c}
\displaystyle
\mathbb P\{|\xi_y^\eps(t_{j+1}-t_j)-(\gamma(\pi(t_{j+1}))-\gamma(\pi(t_j)))|\leq \delta_1\}\\[2mm]
\displaystyle
\leq
\exp\big[{\textstyle -\frac{t_{j+1}-t_j}\eps \big\{L\big(\frac{\gamma(t_{j+1})-\gamma(t_j)}{t_{j+1}-t_j}\big)}
-\varkappa\big\}\big]\\[2mm]
\displaystyle
=\exp\big({\textstyle -\frac{t_{j+1}-t_j}\eps \big\{L(\dot\gamma_N(t))\big._{t\in(t_j,t_{j+1})}}
-\varkappa\big\}\big)
\end{array}
\end{equation}
for all $y$ such that $|y-\gamma(\pi(t_j))|\leq\delta_1$ and for all sufficiently small $\eps$.
This choice is possible due to Theorem \ref{l_findim_per}.
Considering the Markov property of the process $\xi^\eps(t)$ we deduce from \eqref{upp_perio1} that for all
sufficiently small $\eps>0$
the following inequalities hold:
$$
\begin{array}{c}
\displaystyle
\mathbb P\{\sup\limits_{0\leq t\leq T}|\xi_x^\eps(t)-\gamma(\pi(t))|\leq \delta_1\}\leq
\mathbb P\{|\xi_x^\eps(t_j)-\gamma(\pi(t_j))|\leq \delta_1,\, j=0,\ldots, N\}\\
\displaystyle
\leq \prod\limits_{j=0}^{N-1}\exp\big({\textstyle -\frac{t_{j+1}-t_j}\eps \big\{L(\dot\gamma_N(t))\big._{t\in(t_j,t_{j+1})}}
-\varkappa\big\}\big)\\
\displaystyle
=\exp\big({\textstyle -\frac1\eps\big\{ \int_0^TL(\dot\gamma_N(t))\,dt
-T\varkappa\big\}}\big)\leq\exp\big({\textstyle -\frac1\eps\big\{ \int_0^TL(\dot\gamma(t))\,dt
-(T+1)\varkappa\big\}}\big)
\end{array}
$$
This yields the desired upper bound in \eqref{dop_formulirov}.

\medskip
The lower bound can be obtained in a similar way.
It suffices to combine the statement of Theorem \ref{l_findim_per} with \eqref{per_zazor} and use the Markov property of $\xi^\eps(\cdot)$. Indeed, for any $\delta_0>0$ and $\varkappa>0$ we choose the corresponding
$\delta>0$ and $\delta_1>0$ so that \eqref{per_zazor} holds
 and
\begin{equation}\label{lowe_perio1}
\begin{array}{c}
\displaystyle
\mathbb P\{|\xi_y^\eps(t_{j+1}-t_j)-(\gamma(t_{j+1})-\gamma(t_j))|\leq \delta_1\}\\[2mm]
\displaystyle
\geq
\exp\big[{\textstyle -\frac{t_{j+1}-t_j}\eps \big\{L\big(\frac{\gamma(t_{j+1})-\gamma(t_j)}{t_{j+1}-t_j}\big)}
+\varkappa\big\}\big]\\[2mm]
\displaystyle
=\exp\big({\textstyle -\frac{t_{j+1}-t_j}\eps \big\{L(\dot\gamma_N(t))\big._{t\in(t_j,t_{j+1})}}
+\varkappa\big\}\big)
\end{array}
\end{equation}
for all $y$ such that $|y-\gamma(t_j)|\leq\delta_1$ and all sufficiently small $\eps>0$.
 Then 
considering the statement of  Proposition \ref{p_no_zigzag} we have
$$
\mathbb P\{\sup\limits_{0\leq t\leq T}|\xi_x^\eps(t)-\gamma(t)|\leq \delta_0\}
$$
$$
\geq\mathbb P\big\{ \{|\xi_x^\eps(t_j)-\gamma(t_j)| \leq\delta_1,\,j=0,\ldots,N\}\cap
 \{\sup\limits_{0\leq t\leq T}|\xi_x^\eps(t)-\gamma(t)| \leq\delta_0\}\big\}
$$
$$
\geq\mathbb P\big\{ \{|\xi_x^\eps(t_j)-\gamma(t_j)| \leq\delta_1,\,j=0,\ldots,N\}-\exp\big({\textstyle-\frac M\eps}\big)
$$
\begin{equation}\label{lolobou}
\geq \prod\limits_{j=0}^{N-1}\exp\big({\textstyle -\frac{t_{j+1}-t_j}\eps \big\{L(\dot\gamma_N(t))\big._{t\in(t_j,t_{j+1})}}
+\varkappa\big\}\big)-\exp\big({\textstyle-\frac M\eps}\big)
\end{equation}
$$
\geq\exp\big({\textstyle -\frac1\eps\big\{ \int_0^TL(\dot\gamma_N(t))\,dt
+T\varkappa\big\}}\big)\geq\exp\big({\textstyle -\frac1\eps\big\{ \int_0^TL(\dot\gamma(t))\,dt
+T\varkappa\big\}}\big);
$$
here we have also used that fact that $M=I(\gamma)+1$.
This completes the proof of the lower bound in \eqref{dop_formulirov}.

\medskip
In order to justify the large deviation principle we need one more estimate.
Recall that for any $s\in\mathbb R$  the symbol $\Phi(s)$ denotes
$
\Phi(s)=\{\gamma(\cdot)\in\mathbf{D}([0,T];\mathbb R^d)\,:\, I(\gamma)\leq s,\,\gamma(0)=x\}.
$
Observe that the set $\Phi(s)$ consists of absolutely continuous curves and, according to Lemma \ref{l_compa_in_c}, this set is compact.
\begin{lemma}\label{l_up_gener}
For any $s\in\mathbb R$, any $\varkappa>0$ and any $\delta_0>0$ for all
sufficiently small $\eps>0$   the following inequality holds:
\begin{equation}\label{oc_upgene}
\mathbb P\{\mathrm{dist}(\xi^\eps_x(\cdot),\Phi(s))>\delta_0\}\leq \exp\big\{{\textstyle -\frac{s-\varkappa}\eps}\big\}.
\end{equation}
\end{lemma}
\begin{proof}
For any trajectory $\xi_x^\eps(\cdot)$ and any $\delta=\frac TN$, $N\in\mathbb Z^+$  denote
by $\gamma^\eps_{\delta,\omega}(t)$  a piece-wise linear function such that
$$
\gamma^\eps_{\delta,\omega}(j\delta)=\xi_x^\eps(j\delta),\quad j=0,\,1\,\ldots,N;
$$
the argument $\omega$ indicates that $\gamma^\eps_\delta(\cdot)$ is a random function, in what follows
the dependence on $\omega$ is not indicated explicitly.  We choose $\delta>0$ such that
 $$
 |\gamma(t')-\gamma(t'')|\leq \frac14\delta_0,\quad \hbox{if }|t'-t''|\leq\delta\ \hbox{ and }I(\gamma(\cdot))\leq s,
 $$
 and
 $$
 \min\limits_{\phi\in S^{d-1}}\big\{\delta L\big(\frac{\delta_0\phi}{2\delta}\big)\big\}\geq s+1.
 $$
Denote by $\mathcal{E}_-$ and   $\mathcal{E}_+$ the events
$$
\begin{array}{c}
\displaystyle
\mathcal{E}_-=\{\xi^\eps_x(\cdot)\not\in \Phi_{\delta_0}(s),\, I(\gamma^\eps_\delta)< s\},
\qquad
\mathcal{E}_+=\{\xi^\eps_x(\cdot)\not\in \Phi_{\delta_0}(s),\, I(\gamma^\eps_\delta)\geq s\},
\end{array}
$$
where $\Phi_{\delta_0}(s)=\{\gamma(\cdot)\in\mathbf{D}([0,T];\mathbb R^d)\,:\,\mathrm{dist}(\gamma,\Phi(s))
\leq \delta_0\}$.
By Proposition \ref{p_no_zigzag} for all sufficiently small $\eps>0$ we have
\begin{equation}\label{esti_e-}
\mathbb P(\mathcal{E}_-)\leq\mathbb P\big\{\sup\limits_{0\leq t\leq T}
|\xi^\eps_x(t)-\gamma^\eps_\delta(t)|\geq\delta_0\big\}
\leq \exp\big({\textstyle -\frac {s+1}\eps}\big).
\end{equation}

Consider a $(Nd)$-dimensional vector $\{\xi^\eps_x((j+1)\delta)-\xi^\eps_x(j\delta)\}_{j=1}^{N-1}$.
By Theorem \ref{l_findim_per} taking into account Markov property of $\xi^\eps(\cdot)$ we deduce that
the family of random vectors $\{\xi^\eps_x((j+1)\delta)-\xi^\eps_x(j\delta)\}_{j=0}^{N-1}$ satisfies for any
$\lambda_0,\ldots,\lambda_{N-1}\in\mathbb R^d$
the following relation
$$
\lim\limits_{\eps\to0}\eps\log\mathbb E\Big\{\exp\Big[\sum\limits_{j=0}^{N-1}\lambda_j\cdot
(\xi^\eps_x((j+1)\delta)-\xi^\eps_x(j\delta))\Big]\Big\}=\sum\limits_{j=0}^{N-1}\delta h(\lambda_j).
$$
By the G\"artner-Ellis theorem this implies the upper large deviation bound with the rate function
$$
L_\delta(p_1)+L_\delta(p_2)+\ldots+L_\delta(p_N),\quad p_j\in\mathbb R^d,
$$
where $L_\delta(p)=\delta L\big(\frac p\delta\big)$, as was defined in \eqref{def_ham_lag}.
For an arbitrary piece-wise linear function $\gamma$ corresponding to the partition  $\{j\delta\}_{j=0}^N$ we have
$$
I(\gamma)=\sum\limits_{j=0}^{N-1}L_\delta(\gamma((j+1)\delta)-\gamma(j\delta)).
$$
Therefore, by the G\"artner-Ellis theorem, for sufficiently small $\eps>0$ we have
$$
\begin{array}{c}
\displaystyle
\mathbb P(\mathcal{E}_+)\leq\mathbb P\big\{ I(\gamma^\eps_\delta)\geq s\big\}\\
\displaystyle
=\mathbb P\Big\{ \sum\limits_{j=0}^{N-1}L_\delta\big(\xi^\eps_x((j+1)\delta)-\xi^\eps_x(j\delta)\big)\geq s\Big\}
\leq
\exp\big({\textstyle -\frac{s-\varkappa}\eps}\big).
\end{array}
$$
Combining this estimate with \eqref{esti_e-} yields the desired statement.
\end{proof}

From the proof of Lemma \ref{l_up_gener} it follows that for any $s_0>0$ inequality \eqref{oc_upgene} holds
uniformly in $s\in [0,s_0]$, that is for any $\varkappa>0$ and $\delta_0>0$ there exists $\eps_0>0$ such that
 \eqref{oc_upgene} holds for all $\eps\leq\eps_0$ and all $s\leq s_0$.

 It is then well known, see for instance \cite{FW}, that  the lower bound in \eqref{lolobou} and Lemma \ref{l_up_gener} imply the large deviation principle stated in
  Theorem.
\end{proof}

\section{Environments with slowly varying characteristics}
\label{s_slowly}

In this section we consider the case of environments whose characteristics $\Lambda(x,y)$ do not depend on the
fast variables i.e. $\Lambda$ is a continuous function on $\mathbb R^{2d}$ for which condition \eqref{lm}
is fulfilled.   Our approach in this section is somehow inspired by the small perturbations arguments used in the previous works, in particular in the Wentzell-Freidlin theory, see \cite{FW}.
However,
the results from  these works do not apply directly to the operators considered in the present paper
and require some adaptation.

Under the assumptions of this section the generator of  $\xi^\varepsilon(t)$ takes the form
\begin{equation}\label{A-nonosc}
A^\varepsilon u (x) = \frac{1}{\varepsilon^{d+1}} \int_{\mathbb{R}^d} a(\frac{x-y}{\varepsilon}) \Lambda(x,y) (u(y) - u(x)) dy, \; u \in L^2(\mathbb{R}^d),
\end{equation}
$\varepsilon>0$ is a small parameter, and the convolution kernel $a(z)$ in \eqref{A-nonosc} satisfies conditions
\eqref{M1}--\eqref{M2-bis} introduced in the previous section.

\begin{remark}
According to Corollary \ref{cor_eff_drift} below
the Markov jump process  $\xi^\varepsilon(t)$ is a small random perturbation of a deterministic trajectory determined by an ordinary differential equation $\dot{x} = b(x)$ with
$$
b(x) = -\Lambda(x,x) \int a(z)\, z \, dz.
$$
\end{remark}


\subsection{Markov process with slow variables}

We turn now to the case of non-constant $\Lambda(x,y)$ that does not depend on the fast variables and recall that the function $\Lambda(x,y)$ is continuous in both variables and satisfies condition \eqref{lm}. Since $\Lambda$ does not depend on the fast variables,
condition \eqref{cond_conti} can be replaced with the following continuity condition
\begin{equation}\label{cond_conti_bis}
\Lambda(x,y) \quad\hbox{is continuous on }\mathbb R^d\times
\mathbb R^d.
\end{equation}
After changing variables $\tilde x = \frac{x}{\varepsilon}$ the operator $A^\varepsilon$ in \eqref{A-nonosc} takes the form
\begin{equation}\label{A-nonosc-1}
\tilde A^\varepsilon u (\tilde x) = \frac{1}{\varepsilon} \int_{\mathbb{R}^d} a(\tilde x- \tilde y) \Lambda(\varepsilon \tilde x, \varepsilon \tilde y) \big( u(\tilde y) - u(\tilde x) \big) d \tilde y.
\end{equation}
The  Hamiltonian $H(x,\lambda)$ and the Lagrangian $L(x, \zeta)$ are introduced in this case as follows:
\begin{equation}\label{specX}
H(x,\lambda) = \Lambda(x,x) \Big( \int a(z) e^{-\lambda z} dz -1 \Big) = \Lambda(x,x) H(\lambda),
\end{equation}
\begin{equation}\label{Lege}
L(x, \zeta) =  \sup\limits_\lambda \big\{ \lambda \zeta -  \Lambda(x,x) H(\lambda)\big\}
=\Lambda(x,x) L\big({\textstyle \frac \zeta{\Lambda(x,x)}}\big).
\end{equation}

Observe that the function $L(x,\zeta)$ is continuous and non-negative on $\mathbb R^d\times\mathbb R^d$. Moreover,  it is smooth and strictly convex in $\zeta \in \mathbb{R}^d$.
The corresponding {\sl rate function} $I_\Lambda$ is defined by
$$
I_\Lambda(\gamma(\cdot)) = \left\{\begin{array}{ll}
\displaystyle
\int_0^T   {\textstyle L\big(\gamma(t),\dot\gamma(t)\big) dt},&\hbox{if }\gamma\ \hbox{is absolutely continuous,}\\[3mm]
+\infty, &\hbox{otherwise}.
\end{array}
\right.
$$
\begin{theorem}\label{th_slowvar}
Under assumptions \eqref{M1}--\eqref{M2_conseq}, \eqref{lm} and \eqref{cond_conti_bis} the family
of processes $\{\xi^\eps(t),\,0\leq t\leq T\}$ satisfies, as $\eps\to0$, the large deviation principle in the Skorokhod space $\mathbf D([0,T];\mathbb R^d)$ with the rate function $I_\Lambda(\cdot)$.
\end{theorem}
\begin{proof}

\medskip
Consider  an absolutely continuous curve $\gamma(\cdot)$ such that
$$
I_\Lambda(\gamma)=\int_0^T L_{\Lambda(\gamma(t),\gamma(t))}(\dot\gamma(t))\,dt<+\infty.
$$
We first justify the upper bound.   For any $N\geq 2$ denote by $\gamma_N$ a piece-wise linear
interpolation of $\gamma$ such that $\gamma_N(t_j)=\gamma(t_j)$ with $t_j=j\frac TN$, $j=0,1,\ldots,N$,
and by $\widehat\gamma_N$ the corresponding piece-wise constant interpolation, $\widehat\gamma_N(t)=
\gamma(t_j)$ for $t\in [t_j,t_j+\frac1N)$.
For any $\varkappa>0$ there exists $\delta>0$ such that for all $N\geq\delta^{-1}$ we have
$$
\int_0^T{L}\big._{\Lambda(\widehat\gamma_N(t),\widehat\gamma_N(t))}(\dot\gamma_N(t))\,dt>
I_\Lambda(\gamma)-\varkappa.
$$

Denote by $\nu(s)$ the modulus of continuity of $L(x,y)$ in $1$-neighbourhood of the curve $\gamma$.
Since $ \min\limits_{\phi\in S^{d-1}}r^{-1}L(r\phi)$ tends to infinity as $r\to\infty$, there exists a function $\delta_0(\delta)>0$ such that
$\delta_0(\delta)\to 0$ as $\delta\to 0$, and $\min\limits_{\phi\in S^{d-1}}\big\{\delta L(\frac {r\phi}\delta)\,:\,r\geq \delta_0\big\}\to\infty$.
It is  then clear that for any sufficiently small $\delta>0$ there exists $\delta_1(\delta)>0$ such that
$$
\Big|\sum_{j=0}^{N-1} \delta L\big._{(\Lambda(\gamma(t_j),\gamma(t_j))+\nu(\delta_0))}\big(\frac{x_{j+1}-x_j}\delta\big)-
\int_0^T{L}\big._{\Lambda(\widehat\gamma_N(t),\widehat\gamma_N(t))}(\dot\gamma_N(t))\,dt\Big|\leq\varkappa,
$$
if $|x_j-\gamma(t_j)|\leq\delta_1$, $j=0,\ldots,N $.

Consider a Markov process $\xi_x^{\eps,N}(t)$, $0\leq t\leq T$, whose generator on the interval $[t_j,t_j+\frac1N)$
is
$$
A_{t_j}^\eps v(x)=
\frac1{\eps^{d+1}}\int_{\mathbb R^d}(\Lambda(\gamma(t_j),\gamma(t_j))+\nu(\delta_0) ) a\big(\frac{x-y}\eps\big)(v(y)-v(x))dy,\quad
j=0,1,\ldots, N.
$$
We also define a Markov process $\widetilde\xi_x^{\eps,N}(t)$, $0\leq t\leq T$ such that its generator on the interval $[t_j,t_j+\frac1N)$ reads
$$
\widetilde A_{t_j}^\eps v(x)=
\frac1{\eps^{d+1}}\int_{\mathbb R^d}(\Lambda_{t_j}(x,y ) a\big(\frac{x-y}\eps\big)(v(y)-v(x))dy,\quad
j=0,1,\ldots, N,
$$
with $\Lambda_{t_j}(x,y)=\Lambda(\gamma(t_j),\gamma(t_j))+\nu(\delta_0)$  if $|x|+|y|\leq\delta_0$,
and $\Lambda_{t_j}(x,y)=\Lambda(x,y)$ otherwise. \\
Using the inequality similar to that in \eqref{compar_max} we conclude that for any $x$ such that
$|x-\gamma(t_j)|\leq\delta_1$ the density $\widetilde q^{\eps,N}(t,x,y)$ of the process $\widetilde\xi_x^{\eps,N}(t)$, $\widetilde\xi_x^{\eps,N}(t_j)=x$,
on the set $\{(y,t)\,:\,t_j\leq t\leq t_j+\delta,\ |y-\gamma(t_j)|>2\delta_0\}$ does not exceed
$\exp\big(-\frac\delta\eps L\big._{\Lambda^+}(\frac{\delta_0}\delta)\big)$.
Denote by $q^{\eps,N}(t,x,y)$  the density of the process $\xi_x^{\eps,N}(t)$, $\xi_x^{\eps,N}(t_j)=x$.
Straightforward computations show that the difference $\Xi^\eps(t,x,y)=\widetilde q^{\eps,N}(t,x,y)-q^{\eps,N}(t,x,y)$
satisfies on the interval $(t_j,t_j+\delta)$ the equation
$$
\partial_t\Xi^\eps=\int_{\mathbb R^d}[\Lambda(\gamma(t_j),\gamma(t_j))+\nu(\delta_0) ]
a\big({\textstyle \frac{y-z}\eps}\big)(\Xi^\eps(t,x,z)-(\Xi^\eps(t,x,y))\,dz+R^\eps(t,x,y)
$$
 with
 $$
 |R^\eps(t,x,y)|\leq\exp\Big(-\frac\delta{2\eps} L\big._{\Lambda^+}(\frac{\delta_0}\delta)\Big)
 $$
 and $\Xi^\eps(t_j,x,y)=0$.

With the help of the standard a priori estimates this yields
$$
\|\widetilde q^{\eps,N}(t,x,y)-q^{\eps,N}(t,x,y)\|_{L^2(\mathbb R^d)}\leq
\exp\Big(-\frac\delta{4\eps} L\big._{\Lambda^+}(\frac{\delta_0}\delta)\Big).
$$
We choose $\delta>0$ in such a way that $\frac\delta4 L\big._{\Lambda^+}(\frac{\delta_0}\delta)>M$
with $M=I_\Lambda(\gamma)+1$.
Combining the above estimates we obtain
$$
\mathbb
P\big\{\mathrm{dist}(\xi^\eps(\cdot),\gamma(\cdot))\leq\delta_1\big\}
\leq \exp\big({\textstyle \frac{2\nu(\delta_0)T}\eps}\big)
{\displaystyle
\mathbb P\big\{\mathrm{dist}(\widetilde \xi^{\eps,N}(\cdot),\gamma(\cdot))\leq\delta_1\big\}}
$$
$$
\leq
\exp\big({\textstyle \frac{2\nu(\delta_0)T}\eps}\big)\Big[{\displaystyle
\mathbb P\big\{
\{|\xi^{\eps,N}(t_j)-\gamma(\pi(t_j))|\leq\delta_1,\,j=0,\ldots,N\}\big\}}+
\exp\big({\textstyle -\frac M\eps}\big)
\Big]
$$
\begin{equation}\label{up_traj_slow}
\leq\exp\big({\textstyle \frac{2\nu(\delta_0)T}\eps\big)
\exp\bigg[\frac{(1+o(1))}\eps}{\displaystyle
\bigg(-\int_0^T{L}\big._{\Lambda(\widehat\gamma_N(t),\widehat\gamma_N(t))}(\dot\gamma_N(t))\,dt+
\varkappa\bigg)}\bigg]
\end{equation}
$$
\leq\exp\big({\textstyle \frac{2\nu(\delta_0)T}\eps\big)
\exp\big[\frac{(1+o(1))}\eps\
\big(-I_\Lambda(\gamma)+
2\varkappa\big)\big],}
$$
where $o(1)$ tends to zero as $\eps\to0$. This implies the desired upper bound.

\medskip
We turn to the lower bound. Here we introduce $\delta_1=\delta_1(\delta)$ and
$\delta_0=\delta_0(\delta)$ in such a way that
$$
\Big|\sum_{j=0}^{N-1} \delta L\big._{(\Lambda(\gamma(t_j),\gamma(t_j))-\nu(\delta_0))}\big(\frac{x_{j+1}-x_j}\delta\big)-
\int_0^T{L}\big._{\Lambda(\widehat\gamma_N(t),\widehat\gamma_N(t))}(\dot\gamma_N(t))\,dt\Big|\leq\varkappa,
$$
Define Markov processes $\xi_{-,x}^{\eps,N}(t)$ and $\widetilde\xi_{-,x}^{\eps,N}(t)$, $0\leq t\leq T$, whose generators on the interval $[t_j,t_j+\frac1N)$ read, respectively,
$$
A_{-,t_j}^\eps v(x)=
\frac1{\eps^{d+1}}\int_{\mathbb R^d}[\Lambda(\gamma(t_j),\gamma(t_j))-\nu(\delta_0) ] a\big(\frac{x-y}\eps\big)(v(y)-v(x))dy,\quad
j=0,1,\ldots, N.
$$
and
$$
\widetilde A_{-,t_j}^\eps v(x)=
\frac1{\eps^{d+1}}\int_{\mathbb R^d}(\Lambda_{-,t_j}(x,y ) a\big(\frac{x-y}\eps\big)(v(y)-v(x))dy,\quad
j=0,1,\ldots, N,
$$
with $\Lambda_{-,t_j}(x,y)=\Lambda(\gamma(t_j),\gamma(t_j))-\nu(\delta_0)$  if $|x|+|y|\leq\delta_0$,
and $\Lambda_{t_j}(x,y)=\Lambda(x,y)$ otherwise. \\
By comparison with the process $\xi^{+,\eps}(t)$ one can show that for any $M>0$ for sufficiently small $\delta>0$
we have
$$
\mathbb P\big\{\sup\limits_{t_j\leq t\leq t_j+\delta}|\xi^{\eps}(t)-\xi^{\eps}(t_j)|\geq\delta_0\big\}
\leq \exp\big({\textstyle -\frac M\eps}\big).
$$
Using this inequality and choosing $M=I_\Lambda(\gamma)+2$, in the same way as in the proof of the upper bound we obtain
$$
\mathbb P\big\{\sup\limits_{0\leq t\leq T}|\xi^\eps(t)-\gamma(t)|\leq\delta_0\big\}
\geq \mathbb P\big\{\max\limits_{j}|\xi^\eps(t_j)-\gamma(t_j)|\leq\delta_1\big\}-
{\textstyle \exp\big(-\frac {M-1}\eps\big)}
$$
$$
\geq \exp\big({\textstyle -\frac{2\nu(\delta_0)T}\eps}\big)
{\displaystyle
\mathbb P\big\{\max\limits_{j}|\widetilde \xi^{\eps,N}(t_j)-\gamma(t_j)|\leq\delta_1\big\}}
-{\textstyle \exp\big(-\frac {M-1}\eps\big)}
$$
$$
\geq
\exp\big({\textstyle -\frac{2\nu(\delta_0)T}\eps}\big){\displaystyle
\mathbb P\big\{
\{\max\limits_{j}|\xi^{\eps,N}(t_j)-\gamma(t_j)|\leq\delta_1\}\big\}}-
\exp\big({\textstyle -\frac {M-1}\eps}\big)
$$
$$
\geq\exp\big({\textstyle -\frac{2\nu(\delta_0)T}\eps\big)
\exp\bigg[\frac{(1+o(1))}\eps}{\displaystyle
\bigg(-\int_0^T{L}\big._{\Lambda(\widehat\gamma_N(t),\widehat\gamma_N(t))}(\dot\gamma_N(t))\,dt-
\varkappa\bigg)}\bigg]
$$
$$
\geq\exp\big({\textstyle - \frac{2\nu(\delta_0)T}\eps\big)
\exp\big[\frac{(1+o(1))}\eps\
\big(-I_\Lambda(\gamma)-
2\varkappa\big)\big],}
$$
where $o(1)$ tends to zero as $\eps\to0$. This yields the lower bound.

We should also show that  for any $s\geq 0$, any $\delta_0>0$  and any $\varkappa>0$
\begin{equation}\label{oc_upgene_slowvar}
\mathbb P\{\mathrm{dist}(\xi^\eps_x(\cdot),\Phi(s))>\delta_0\}\leq \exp\big\{{\textstyle -\frac{s-\varkappa}\eps}\big\}
\end{equation}
for all sufficiently small $\eps$.
The proof of this inequality relies on the arguments from the proof of Lemma \ref{l_up_gener} and that of inequality \eqref{up_traj_slow}. One should combine these arguments in a straightforward way. We skip the details.
\end{proof}

 \bigskip
Denote
$$
\zeta (x) = {\rm arg}\min\limits_{p} \, L_\Lambda \big(x, p \big).
$$
It is straightforward to check that
\begin{equation}\label{v0}
\zeta(x) = - \Lambda(x,x) \int_{\mathbb R^d} a(z) \, z \, dz=\Lambda(x,x) \nabla H(\lambda)\big|_{\lambda=0}.
\end{equation}
Letting $\gamma^0_x(t)$
be the solution of the ODE
$$
\dot\gamma(t)=\zeta(\gamma(t)),\quad \gamma(0)=x,
$$
one can deduce from the last theorem the following
\begin{corollary} \label{cor_eff_drift}
For any $x\in\mathbb R^d$
$$
\lim\limits_{\eps\to0}\,\mathbb E \big(\sup\limits_{0\leq t\leq T}|\xi_x^\eps(t)-\gamma^0_x(t)|\big) =0.
$$
\end{corollary}

\section{The general case of locally periodic environment $ \Lambda(x,y, \frac{x}{\varepsilon},  \frac{y}{\varepsilon})$}

In this section we consider the case of the most general locally periodic media. Here we assume that
$\Lambda^\eps(x,y)=\Lambda(x,y,\frac x\eps,\frac y\eps)$, where $\Lambda(x,y,\xi,\eta)$ satisfies
conditions \eqref{cond_perio}--\eqref{lm}.

Here for each $x\in\mathbb R^d$ we introduce a Hamiltonian $H=H(x,\lambda)$  in the same way as
in \eqref{conseq}, $x$ being a parameter. Namely, we set
\begin{equation}\label{conseq_param}
H(x,\lambda) :=\mathtt{s}(A_{x,\lambda}) = \left\{
\begin{array}{l}
\theta(x,\lambda), \; \lambda \in \Gamma(x) \\
- g_{\min}(x), \ \mbox{otherwise};
\end{array}
\right.
\end{equation}
here
$$
A_{x,\lambda}u(z)=
\int_{\mathbb R^d}\Lambda(x,x,z,y)a(z-y)e^{\lambda\cdot(y-z)}u(y)\,dy
- \int_{\mathbb R^d}\Lambda(x,x,z,y)a(z-y)dy\, u(z),
$$
and, for each $x$, we define $\Gamma(x)$, $\theta(x,\lambda)$ and $g_{\min}(x)$  in the same way as in Section \ref{s_perio}.
Then we introduce the corresponding Lagrangian $L(x,\zeta)$.
The main result of this section reads:
\begin{theorem}\label{t_main_gen}
Let conditions \eqref{M1}--\eqref{lm} be fulfilled.  Then the family of processes $\xi^\eps_x(\cdot)$
with the generators $A^\eps$ defined in \eqref{A} satisfies, as $\eps\to0$, the large deviation principle
in the path space $\mathbf{D}([0,T];\mathbb R^d)$;  the corresponding rate function is given by
$$
I(\gamma(\cdot))=\left\{\begin{array}{ll}\displaystyle
\int_0^T L(\gamma(t),\dot\gamma(t))\,dt,& \hbox{\rm if }\gamma\ \hbox{\rm is abs. cont. and }\gamma(0)=x,\\[3mm]
+\infty,& \hbox{\rm otherwise}.
\end{array}
\right.
$$
\end{theorem}
\begin{proof}
The proof relies on combining the statement of Theorem \ref{l_findim_per} and the arguments used
in the proof of Theorem \ref{th_slowvar}. We leave the details to the reader.
\end{proof}

It is interesting to observe that for small $\eps>0$ the process $\xi_x^\eps(\cdot)$ can be interpreted as
a small random perturbation of a deterministic dynamical system defined by the ODE
\begin{equation}\label{determ_curve}
\dot\gamma(t)=\nabla_\lambda H(\gamma(t),0),\qquad \gamma(0)=x.
\end{equation}
\begin{corollary}
For any $x\in\mathbb R^d$
$$
\lim\limits_{\eps\to0}\mathbb E\big\{\sup\limits_{0\leq t\leq T}|\xi^\eps_x(t)-\gamma_x(t)|\big\}=0,
$$
where $\gamma_x(\cdot)$ is a solution of \eqref{determ_curve}.
\end{corollary}


\begin{thebibliography}{99}

\bibitem{Bal91} Baldi, P. Large deviations for diffusion processes with homogenization and applications. {\it   Ann. Probab.},  {\bf 19}, (1991),  p.509--524.

\bibitem{Bor67} Borovkov, A.A. Boundary problems for random walk and large deviations on the functional spaces,
Theory Probab. Appl. 12, (1967), p. 635--654.

\bibitem{DeZe} Dembo, A., Zeitouni, O. {\it Large Deviations Techniques and Applications}, (1998), Springer.


\bibitem{EdPoSt72} Edmunds, D. E., Potter, A. J. B. and Stuart, C. A.Non-compact positive operators,
 {\it Proc. R. Soc. Lond. A.}, {\bf 328}, (1972), p. 67--81.

\bibitem{EnNa99} Engel, K.-J., Nagel, R. {\it One-parameter Semigroup for Linear Evolution Equations},
(1999), Springer.

\bibitem{FeKu2005} Jin Feng, J., Kurtz, T.G. {\it Large Deviations for Stochastic Processes},
Mathematical Surveys and Monographs, vol. {\bf 131}, (2006) , AMS, Providence.

\bibitem{Fr} Freidlin, M.I. 
Averaging principle and large
deviations theorems. {\it Russian Math. Surveys},  {\bf 33}(5), (1978),  p. 117--176.

\bibitem{FrSow} Freidlin, M.I., Sowers, R.B. A comparison of homogenization and large deviations, with
applications to wavefront propagation,  {\it Stoch. Proc. Appl.}, {\bf 82}, (1999),  p.23--52.

\bibitem{FW} Freidlin, M.I. and Wentzell, A.D.{\it Random
Perturbations of Dynamical Systems.}, (1984) , N.Y., Springer-Verlag.

\bibitem{WF_70} Freidlin, M.I. and Wentzell, A.D. On small random perturbations of dynamical systems,
{\it Russian Math. Surv.}, {\bf 25}:(1), (1970),  p.1--55.

\bibitem{JKO} Jikov, V.V., Kozlov, S.M., Oleinik, O.A.
{\it Homogenization of Differential Operators and Integral
Functionals.},  (1994), Springer-Verlag.

\bibitem{GKPZ}
 Grigor'yan A., Kondratiev Yu., Piatnitski A. and Zhizhina E.
  Pointwise estimates for heat kernels
of convolution type operators, {\it Proc. London Math. Soc.}, {\bf 117}(4), (2018), p. 849--880,
 dx.doi.org/10.1112/plms.12144.


\bibitem{KLS} Krasnosel'skii, M.A., Lifshits, E.A., Sobolev A.V.
{\it Positive Linear Systems. The Method of Positive Linear Operators.}
Sigma Series in Applied Math., {\bf 5}. , (1989), Heldermann Verlag.

\bibitem{LS} Lynch, J. and Sethuraman, J. Large deviations for processes with
independent increments, {\it Ann. Probab.}, {\bf 15}, (1987), p. 610--627.


\bibitem{M1993} Mogulskii, A.A. Large deviations for processes with independent increments,
{\it Ann. Probab.}, {\bf 21}(1),  (1993), p. 202--215.

\bibitem{M2016} Mogulskii, A.A. The large deviation principle for a compound Poisson process,
{\it Mat. Tr.}, {\bf 19}(2), (2016), p. 119--157.

\bibitem{PZh2019} A. Piatnitski, E. Zhizhina, Homogenization of  biased convolution type operators,
{\it Asymptotic Analysis}, {\bf 115}(3-4), (2019), p. 241--262, doi:10.3233/ASY-191533.

\bibitem{Pukh94} Puhalskii,  A.,  The method of stochastic exponentials for large deviations.,
{\it Stoch. Proc. Appl.}, {\bf 54}, (1994), p. 45--70.


\bibitem{Str} Stroock, D.W. {\it An Introduction to the Theory of Large
Deviations.},  (1984), N.Y., Springer.


\bibitem{We86} Wentzell, A.D.{\it Limit Theorems on Large Deviations for Markov Stochastic Processes}, (1990),    Springer.

\bibitem{ZhP} Zhizhina E.A., Piatnitski A.L., On the behaviour at infinity of solutions to nonlocal parabolic type problems,
{\it The Bulletin of Irkutsk State University. Series Mathematics}, {\bf 30},  (2019), pp. 99--113,
doi.org/10.26516/1997-7670.2019.30.99.



\end{thebibliography}
\end{document}